\documentclass[a4paper, english, 9pt]{amsart}

\usepackage{amsthm,amssymb,url,hyperref}
\usepackage{fullpage}
\usepackage{MnSymbol}
\usepackage[all]{xy}
\usepackage{tikz}
\usepackage{subfig} 
\usepackage{caption}
\usepackage{forest}
\usepackage{color}

\usepackage{stmaryrd}

\theoremstyle{plain}
\newtheorem{theorem}{Theorem}[section]

\newtheorem{corollary}[theorem]{Corollary}

\newtheorem{lemma}[theorem]{Lemma}

\newtheorem{proposition}[theorem]{Proposition}

\theoremstyle{definition}
\newtheorem{definition}[theorem]{Definition}

\newtheorem{example}[theorem]{Example}
\newtheorem{remark}[theorem]{Remark}

\def\ker#1{\mathrm{ker}(#1)}

\def\aut#1{\mathrm{Aut}(#1)}

\def\aff#1{\mathrm{Aff}#1}

\def\lmlt{\mathrm{LMlt}}
\def\dis{\mathrm{Dis}}
\def\sym{\mathrm{Sym}}

\def\c#1{\mathrm{con}_{#1}}
\def\Q{\mathcal{Q}_{\mathrm{Hom}}}

\def\setof#1#2{\{#1\,\mid\,#2\}}

\def\Q{\mathcal Q}

\def\LSS{\mathcal{LSS}}

\def\cg#1{\equiv_\alpha}
\newcommand*\xbar[1]{%
   \hbox{%
     \vbox{%
       \hrule height 0.5pt % The actual bar
       \kern0.5ex%         % Distance between bar and symbol
       \hbox{%
         \kern-0.1em%      % Shortening on the left side
         \ensuremath{#1}%
         \kern-0.1em%      % Shortening on the right side
       }%
     }%
   }%
} 

\title{Connected quandles of size $pq$ and $4p$}

\author{Marco Bonatto}
\address[M. Bonatto]{IMAS--CONICET and Universidad de Buenos Aires, 
		Pabell\'on~1, Ciudad Universitaria, 1428, Buenos Aires, Argentina}
\email{marco.bonatto.87@gmail.com}

\begin{document}

%\thanks{Research partially supported by the GA\v CR grant 13-01832S.}

%\keywords{Medial quasigroup, quasigroup affine over abelian group, classification of quasigroups, enumeration of quasigroups.}

%\date{\today}

\begin{abstract}
	We classify all non-simple connected quandles of size $pq$ and $4p$ where $p,q$ are primes, as a special family of locally strictly simple quandles (i.e. quandles for which all proper subquandles are strictly simple). In particular we classify all latin quandles of size $pq$ and $4p$ and we show that latin quandles of size $8p$ are affine.
\end{abstract}

\maketitle
%\tableofcontents

\section*{Introduction}

Quandles are binary algebras related to knot invariants \cite{J,Matveev} and they provide a special class of solutions to the set-theoretical quantum Yang-Baxter equation \cite{EGS,ESS}. \\
Quandles have been studied using both group and module theory and in a recent paper \cite{CP} we developed a universal algebraic approach, in particular towards the connection between the properties of the \textit{displacement group} and its subgroups and properties of congruences as {\it abelianness} and {\it centrality} in the sense of commutator theory \cite{comm}.

In \cite{Principal} the class of finite {\it strictly simple quandles}, i.e. quandles with no proper subquandles has been investigated and it has been proved that it coincides with the class of simple abelian quandles (in the sense of \cite{comm,CP}). They can be described using affine representations over finite fields and their lattice of subquandles containing a given element is the two elements lattice, i.e. it has height $1$. In this paper we are studying finite quandles such that every proper subquandle is strictly simple and we call them \textit{locally strictly simple} (LSS) quandles. Equivalently LSS quandles are quandles for which the height of the lattice of subquandles containing a given element is at most $2$. Strictly simple quandles are examples of LSS quandles, but there exist non-trivial examples as non-simple quandles of size 28 of the RIG database of GAP \cite{RIG}.\\
 Connected quandles of size $p$ \cite{EGS}, $p^2$ \cite{Grana_p2}, and the {\it affine} ones of size $p^3$ and $p^4$ \cite{Hou} where $p$ is a prime have been classified. All such quandles are affine and the first class of non-affine quandles which has been classified is the class of connected quandles of size $p^3$ \cite{GB}. Using the results on LSS quandles, we can provide the classification of two more families of non-affine quandles, namely the classification of non-simple connected quandles of size $pq$ and $4p$ where $p,q$ are primes.
 
In Section \ref{Sec: prelim} we collect some basic facts about quandles and we show some preliminary results.

In Section \ref{Sec:LSS} we characterize the congruence lattice of non-simple connected LSS quandles and we show that this class of quandles coincide with the class of {\it extensions} of a strictly simple quandle by a strictly simple quandle (i.e. quandles having a congruence with strictly simple factor and strictly simple blocks).\\
First we show that connected subdirectly reducible LSS quandles are direct products of strictly simple quandles and then we focus on the subdirectly irreducible case.\\
Finite connected abelian quandles are {\it polynomially equivalent} to modules over the ring of Laurent polynomial \cite{Principal}[Section 2.4], so we provide a module-theoretical description of abelian connected LSS quandles. 

In Section \ref{non-ab LSS} we investigate non-abelian LSS quandles employing a case-by-case discussion using the equivalence relation $\sigma_Q$ (already defined in \cite{CP} in connection with central congruences, see definition \eqref{sigma}), defining the classes of quandles $\LSS(p^m,q^n,\sigma_Q)$, for $p,q$  primes, as the classes of quandles which are extensions of a strictly simple quandle of size $q^n$ by a strictly simple quandle of size $p^m$ and with given $\sigma_Q$. We completely characterize the displacement groups of such quandles and we compute some constructions involving special and extraspecial $p$-groups. Then we focus on the case $n=1$ and $m=1$, with the goal to describe quandles of size $pq$ and $4p$ and as a byproduct we obtain that latin quandles of size $8p$ where $p$ is a prime are affine.

 The first step in the classification of non-simple connected quandles of size $pq$ and $4p$ where $p$ and $q$ are primes, is to show that they are all LSS quandles, except for a few small primes (see Section \ref{Sec:pq} and Section \ref{sec 4p} respectively). \\
 Connected affine quandles of size $pq$ and $4p$ decompones as direct products of quandles prime power size. We show that there exists only two non-affine non-simple connected quandles of size $pq$ whenever $p^2=1 \pmod{q}$ (Theorem \ref{iso_class_invol}) and two non-simple non-affine connected quandles of size $4p$ whenever $p=1\pmod 3$ (Theorem \ref{iso class 4p}) and according to Proposition \ref{final on ab ext} the latter can not be obtained by {\it abelian extensions}. In particular such quandles are all the non-affine latin quandles of size $pq$ and $4p$ respectively.\\ 
 All such quandles are latin and using the one-to-one correspondence between Bruck Loops of odd order and involutory latin quandles we can show that there is only one Bruck Loop of order $pq$, as already proved in \cite{BruckPetr}. \\
The classification strategy is similar to the one adopted in \cite{GB}. We first characterize the congruence lattice of the LSS quandles and then we identify the displacement group. Then we compute the isomorphism classes which are in one-to-one correspondence with conjugacy classes of automorphisms of such groups (see Theorem \ref{iso theorem1}). All the group theoretical results and the technical computations useful to this end are collected in the appendix \ref{Sec:appendix}.

Finally in Section \ref{non-connected} we give a description of finite non-connected LSS quandles.

\subsection*{Notation and terminology}

An algebra is a set $A$ endowed with a set of operations. We denote by $Sg(X)$ the subalgebra generated by $X\subseteq A$ and its automorphism group by $\aut{A}$. A congruence of $A$ is an equivalence relation which respects the algebraic structure. Congruences of an algebra form a lattice denoted by $Con(A)$ with minimum the identity relation $0_Q=\setof{(a,a)}{a\in Q}$ and maximum $1_Q=A\times A$. By virtue of the second homomorphism thereom homomorphic images and congruences are essentially the same. The factor algebra with respect to $\alpha\in Con(A)$ is denoted by $A/\alpha$. We refer the reader to \cite{UA} for further details. \\
%The congruences of the lower central series and of the derived series of $A$ are denoted respectively by $\setof{\gamma_n(A)}{n\in \mathbb{N}}$ and by $\setof{\gamma^{n}(A)}{n\in \mathbb{N}}$. The center of $A$ is denoted by $\zeta_A$ 
Am algebra is nilpotent (resp. solvable) if there exists a chain of congruences
$$\alpha_0=0_A\leq \alpha_1\leq\ldots\leq \alpha_n=1_A$$
such that $\alpha_{i+1}/\alpha_i$ is central (resp. abelian) in $A/\alpha_i$ for every $0\leq i\leq n-1$ (see \cite{CP, comm} for the formal definitions).
 
%We denote by $Sg(S)$ the smallest subalgebra containing the subset $S\subseteq A$ and $Cg(S)$ the smallest congruence containing the set of pairs $S\subseteq A\times A$. 
The group operations are denoted by juxtaposition or by $+$ in the abelian case. The inner automorphism of an element $g $ of a group $G$ is denoted by $\widehat{g}$ and $N_G(S)$ (resp. $C_G(S)$) is the normalizer (resp. centralizer) of a subset $S\subseteq G$. The core of a subgroup $H$ (i.e. the biggest normal subgroup of $G$ contained in $H$) is denoted by $Core_G(H)$. If $G$ acts on a set $Q$ we denote  the set-wise point-stabilizer of a subset $S$ by $G_S=\setof{g\in G}{g(S)=S}$ and the point-wise stabilizer of $a$ in $G$ by $G_a$. The orbit of $a$ under the action of $G$ will be denoted by $a^G$. A group is called semiregular if $G_a=1$ for every $a\in G$ and regular if it semiregular and transitive. Note that the pointwise-stabilizer of a transitive group is normal if and only if it is trivial (the stabilizers of a transitive group are all conjugate).\\
If $f$ is an automorphism of a group $G$ we denote $Fix(f)$ the set of its fixed points $\setof{a\in G}{f(a)=a}$. If $H$ is an $f$-invariant subgroup of $G$ we denote by $f_H$ the induced mapping on $G/H$ defined as $f_H(gH)=f(g)H$.

\subsection*{Acknowledgments}
The author wants to thank Daniele Toller and Giuliano Bianco for the useful comments and remarks on the drafts of the present work.

\section{Preliminary results}\label{Sec: prelim}
Quandles are idempotent left-distributive left quasigroups, namely a quandle $Q$ is a a binary algebra $(Q,\ast,\backslash)$ satisfying
\begin{eqnarray*}
a\ast (a\backslash b) &=& a\backslash (a\ast b)=b,\\
a\ast( b\ast c)&=&(a\ast b)\ast (a\ast c),\\
a\ast a &=& a.
\end{eqnarray*}
for every $a,b\in Q$. Note that the left multiplication mapping $L_a:b\mapsto a\ast b$ is in the stabilizer of $a$ of the automorphism group of $Q$ for every $a\in Q$.
\begin{example}\label{example2}

(i) Any set $Q$ with the operation $a\ast b=b$ for every $a,b\in Q$ is a \emph{projection quandle}. If $|Q|=n$ we denote such quandle by $\mathcal{P}_n$.

(ii) Let $G$ be a group and $H\subseteq G$ be a subset closed under conjugation. Then $Conj(H)=(H,*)$, where $g*h=g^{-1} h g$ for every $h,g\in H$, is a quandle.

(iii) Let $G$ be a group, $f\in  \aut{G}$ and $H \leq Fix(f)$. Let $G/H$ be the set of left cosets of $H$ and the multiplication defined by
\begin{displaymath}
   aH\ast bH=af(a^{-1}b)H.
\end{displaymath}
%Then $\Q(G,H,f)=(G/H,\ast,\backslash) $ is a quandle, called a \emph{coset} quandle. 
%\item Let $X$ be a quandle and $L_x^{2}=id_X$ for every $x\in X$, then $X$ is called \emph{involutory}.
Then $(G/H,*,\backslash)$ is a quandle denoted by $\Q(G,H,f)$. Such a quandle is called \emph{principal} (over $G$) if $H=1$ and it will be denoted just by $\Q(G,f)$, and it is \emph{affine} (over $G$) if $G$ is abelian, in this case we use the notation $\aff(G,f)$. 
\end{example}
A quandle is called {\it homogeneous} if $\aut{Q}$ acts transitively on $Q$. A quandle $Q$ is homogeneous quandles if and only if admits a representation as in Example \ref{example2}(iii), i.e. $Q=\mathcal{Q}(G,H,f)$ \cite{J}. \\
 The {\it left multiplication} group is the group generated by the left multiplications mappings and it is denoted by $\lmlt(Q)$ and the {\it displacement group} is the group generated by $ \setof{L_a L_b^{-1}}{a,b\in Q}$ and denoted by $\dis(Q)$. A quandle $Q$ is called {\it connected} if $\lmlt(Q)$ (or equivalently $\dis(Q)$) acts transitively on $Q$. \\
If $Q=\mathcal{Q}(G,H,f)$, the action of $\dis(Q)$ is given by the canonical left action of the group $[G,f]=\langle\setof{gf(g)^{-1}}{g\in G}\rangle$ on $G/H$ and $\dis(Q)\cong [G,f]/Core_G(H)$ (see \cite[Section 2]{GB}). According to \cite[Theorem 4.1]{HSV}, connected quandles can be represented over their displacement group as $\Q(\dis(Q),\dis(Q)_a,\widehat{L}_a)$. For connected quandles the representation over the displacement group provides useful information. In particular, since $[\dis(Q),\widehat{L_a}]=\dis(Q)$, the order of the left multiplications of $Q$ coincides with the order of $\widehat{L_a}$. 
 
 \begin{lemma}\label{order of aut and left mult}
 	Let $G$ be a finite group, $f\in \aut{G}$ and $H\leq Fix(f)$. If $G=[G,f]$ then $f$ and $f_H$ have the same order.
 \end{lemma}
 
 \begin{proof}
 	%Note that the order of $f$ is the l.c.m. of the length of the orbits of the generators $\setof{gf(g)^{-1}}{g\in G}$. 
 	Clearly if $f^n(g)=g$ for every $g\in G$ then also $f^n(g)H=H$ for every $g\in G$, so the order of $f_H$ divides the order of $f$. On the other hand if $f^n(g)H=gH$ for every $g\in G$ then $g^{-1}f^n(g)\in H\leq Fix(f)$ and so $f(g^{-1}f^n(g))=f(g^{-1})f^{n+1}(g))=g^{-1}f^n(g)$ for every $g\in G$. Therefore $f^n(gf(g)^{-1})=gf(g)^{-1}$ for every $g\in G$. Then the order of $f$ divides the order of $f_H$ since $G$ is generated by $\setof{gf(g)^{-1}}{g\in G}$.
 \end{proof}
 
The representation of connected quandles over groups is also useful for isomorphism checking. Indeed we have the following:

\begin{theorem}\label{iso theorem1}\cite[Theorem 5.11, Corollary 5.12]{GB}
	Let $Q_i = \mathcal{Q}(G_i,H_i,f_i)$ be a connected quandles for $i=1,2$. If $H_i=Fix(f_i)$ or $H_i=1$ for $i=1,2$ then $Q_1\cong Q_2$ if and only there exists a group isomorphism $\psi:G_1\longrightarrow G_2$ and $f_2=\psi f_1 \psi^{-1}$.
\end{theorem}
For every congruence $\alpha$ of a quandle $Q$ the mapping
\begin{equation}\label{pi alpha}
\pi_\alpha :\dis(Q)\longrightarrow \dis(Q/\alpha),\quad L_{x_1}L_{x_2}^{-1}  \mapsto L_{[x_1]}L_{[x_2]}^{-1}
\end{equation}
is a well-defined surjective group morphism and its kernel is
\begin{displaymath}
\dis^{\alpha}=\setof{h\in \dis(Q)}{[h(a)] =[a], \text{ for every } a\in Q}= \bigcap_{[a]\in Q/\alpha} \dis(Q)_{[a]},
\end{displaymath}
where $\dis(Q)_{[a]_\alpha}=\setof{h\in \dis(Q)}{h([a])= [a]}$ is the set-wise stabilizer in $\dis(Q)$. Note that $\dis(Q)_{[a]_\alpha}=\pi_\alpha^{-1}(\dis(Q/\alpha)_{[a]_\alpha})$. 
%and if $Q/\alpha$ is connected then $\dis^\alpha=Core_{\lmlt(Q)}(\dis(Q)_{[a]})$ for every $a\in Q$. 
%\comment{According to \cite{GB} a quandle $Q$ is connected if and only if $Q/\alpha$ is connected and $\dis(Q)_{[a]_\alpha}$ is transitive on $[a]_\alpha$.}

The representation of connected quandles over the displacement group can be used to get the representation of factors. 
\begin{lemma}\label{rep_for_factors}\cite[Lemma 1.4]{GB}\label{rep for factors}
Let $Q=\Q(G,G_a,f)$ be a connected quandle where $a\in Q$, $G=\dis(Q)$ and $f=\widehat{L_a}$. Then $Q/\alpha\cong \Q(G/\dis^\alpha,\dis(Q)_{[a]_\alpha}/\dis^\alpha,f_{\dis^\alpha})$ for every $\alpha\in Con(Q)$.
\end{lemma}

%In \cite{CP} we show an useful Galois connection between the congruence lattice of a quandle and the lattice 
Let $N\in Norm(Q)=\setof{N\trianglelefteq \lmlt(Q)}{N\leq \dis(Q)}$. The orbit decomposition with respect to the action of $N$ denoted by $\mathcal{O}_N$ and the relation $\c{N}=\setof{(a,b)\in Q^2}{L_a L_b^{-1}}$ are congruences of $Q$ \cite[Lemma 2.6, Lemma 3.4]{CP}. If $\alpha\in Con(Q)$ the {\it displacement group relative to the congruence $\alpha$} is defined analogically to the displacement group as $\dis_\alpha=\langle\setof{L_a L_b^{-1}}{a\, \alpha\, b}\rangle$ (note that relative displacement groups belongs to $Norm(Q)$ and that $[\dis(Q),\dis^\alpha]\leq\dis_\alpha\leq \dis^\alpha$). 
%
%\begin{proposition}\cite[Proposition 3.6]{CP} \label{galois_connection}
	%Let $Q$ be a rack. Then $\alpha\mapsto\dis_\alpha$ and $N\mapsto\c{N}=\setof{(a,b)\in Q\times Q}{L_a L_b^{-1}\in N}$ is a monotone Galois connection between $\Con(N)$ and $Norm(Q)$.
%\end{proposition}
In particular if $\alpha$ is a minimal congruence (with respect to inclusion) then its relative displacement group is minimal in $Norm(Q)$.
\begin{lemma}\label{minimal congruences}
	Let $Q$ be a quandle and $\alpha$ be a minimal congruence. Then $\dis_\alpha$ is a minimal element of $Norm(Q)$. In particular if $\dis_\alpha$ is solvable, then it is abelian.
\end{lemma}

\begin{proof}
	Let $N< \dis_\alpha$. Then $\mathcal{O}_N\leq \alpha$. If equality holds, $b=n(a)$ for some $n\in N$ whenever $a\,\alpha\, b$ and then 
	$$L_a L_b^{-1}=L_a L_{n(a)}^{-1}=\underbrace{L_a n^{-1} L_a^{-1}}_{\in N} n \in N$$ and therefore $\dis_\alpha\leq N$, contradiction. Hence $\mathcal{O}_N=0_Q$ and $N=1$. If $\dis_\alpha$ is solvable, then its derived subgroup is a proper subgroup and it belongs to $Norm(Q)$ and therefore it is trivial i.e. $\dis_\alpha$ is abelian. 
\end{proof}
The properties of congruences as abeliannes and centrality in the sense of \cite{comm} are completely determined by the properties of the relative displacement groups (see \cite[Theorem 1.1]{CP}). The smallest congruences of $Q$ with abelian factor is denoted by $\gamma_Q$ and the center of $Q$, i.e. biggest central congruence is denoted by $\zeta_Q$ (these congruences are the analogous of the derived subgroup and the center of a group in the commutator theory for quandles \cite{comm,CP}).

The equivalence relation $\sigma_Q$ is defined as 
\begin{equation}\label{sigma}
a\, \sigma_Q \, b \quad \text{if and only if}\quad \dis(Q)_a=\dis(Q)_b.\end{equation} 
This relation contains every central congruence of $Q$ \cite[Proposition 5.9]{CP} and its classes are subquandles which are blocks with respect to the action of $\lmlt(Q)$ \cite[Proposition 1.5]{Principal}. Moreover a connected quandle $Q$ is principal if and only if $\sigma_Q=1_Q$.

The structure of groups in $Norm(Q)$ can be investigated using the following Lemma.

\begin{lemma}\label{embedding as quandle}
	Let $Q$ be a connected quandle, $\alpha\in Con(Q)$ such that $Q/\alpha=Sg([a_1],\ldots, [a_n])$ and let $N\in Norm(Q)$ such that $N\leq \dis^\alpha$. Then:
	\begin{itemize}
		\item[(i)] the mapping
		\begin{displaymath}\psi: \dis^\alpha\longrightarrow \prod_{i =1}^n \aut{[a_i]},\quad h\mapsto (h|_{[a_1]},\ldots, h|_{[a_n]}),
		\end{displaymath}
		is an embedding of groups and $N|_{[a_i]_\alpha}\cong N|_{[a_j]_\alpha}$ for every $1\leq i,j\leq n$.
		
		\item[(ii)] If $\alpha=\mathcal{O}_N\leq \sigma_Q$ and $[a]_\alpha$ is a principal connected quandle then $N$ embeds into $\dis([a])^n$.
		%In particular $\dis(Q)_{a_1}$ embeds into $\prod_{i =2}^n \aut{[a_i]}$.
	\end{itemize}
\end{lemma}

\begin{proof}
	(i) The union of the blocks of $a_1,\ldots,a_n$ generates the quandle $Q$. If $\psi(h)=1$ then $h$ fixes a set of generators of $Q$ and therefore $h=1$. Then $\psi$ is an embedding. Let $a_i,a_j\in Q$, $h\in \dis(Q)$ such that $h([a_i])=[a_j]$ and $\psi_i: N\longrightarrow N|_{[a_i]}$. Then $\ker{\psi_j}=h^{-1} \ker{\psi_i} h$ and so the mapping
	\begin{displaymath}
	N|_{[a]_i}\longrightarrow N|_{[a_j]},\quad g|_{[a_i]}\mapsto h^{-1}gh|_{[a]_j}
	\end{displaymath}
	is a well defined group isomorphism.
	
	(ii) The group $N$ is transitive on each block of $\alpha$ and whenever $h(a)=a$ then $h|_{[a]}=1$ since $\alpha\leq \sigma_Q$. Then $N|_{[a]}$ is a regular automorphism group of $[a]$. The blocks are principal connected quandles, so every regular automorphism group is isomorphic to the displacement group of the block, then $H|_{[a]}\cong \dis([a])$ \cite[Proposition 2.1]{Principal}. 
\end{proof}

The set $L(Q)=Conj(\setof{L_a}{a\in Q})$ is a conjugation quandle and $L_Q:Q\to L(Q)$ mapping $a$ to $L_a$ is a surjective quandle morphism. We denote by $\lambda_Q$ the kernel of such homomorphism. A quandle is called {\it faithful } if $L_Q$ is injective, i.e. $\lambda_Q=0_Q$ and {\it latin} if all right multiplications $R_a: b\mapsto b\ast a$ are bijective and in such case we use the notation $R_a^{-1}(b)=b/a$. Latin quandle are faithful and connected. The pointwise stabilizer of an element $a\in Q$ in a faithful quandle is $\dis(Q)_a=Fix(\widehat{L}_a)$. \\ 
We refer the reader to \cite{AG,HSV,J} for further basic results on quandles.

\section{Locally strictly simple quandles}\label{Sec:LSS}
\subsection*{Congruence lattice}
%\comment{better name?}

{\it Strictly simple quandles} are quandles with no proper subquandles (proper subquandles have more than one element). The class of finite strictly simple quandle is well understood as it coincides with the class of all simple abelian quandles, that is the class cointaining all the simple affine quandles and $\mathcal{P}_2$ (see \cite[Section 3.1]{Principal}). Every strictly simple connected quandle is latin and it has an affine representation as $M=\aff(\mathbb{Z}_p^m,f)$ where $f$ acts irreducibly. In particular, finite strictly simple quandles are in one-to-one correspondence with irreducible polymonials over finite fields.
%(in particular the order of $f$ and $p$ are coprime). 
%Up to isomorphism, $M=\aff(\mathbb{F}_q,\lambda)$ where $\mathbb{F}_q$ is the finite field on $q=p^n$ elements and $\lambda\in \mathbb{F}_q^*$. The powers of $M$ are given by
%$$M^n\cong \aff(\underbrace{\mathbb{Z}_p^{m}\times \ldots  \mathbb{Z}_p^m}_{n},\underbrace{f\times\ldots\times	f}_n).$$ 
%Subquandles of powers are subspaces invariant under $f^{\times n}=f\times \ldots\times f$ and so they can be understood as subrepresentations of the cyclic group generated by $f^{\times n}$. According to Maschke theorem, every subrepresentation is a direct product of irreducible representations of $M^n$ and all of them are isomorphic to $M$. Therefore every subquandle of $M^n$ is isomorphic to $M^k$ for some $k\leq n$.\\
%On the other hand $M$ can be understood as $M=\aff(\mathbb{F}_{p^m},\lambda)$. can be interpreted as $\aff(\mathbb{F}_{p^m}^n,\lambda I)$, where $\lambda I$ denote the scalar multiplication by $\lambda$. 
%
%Affine quandles are modules over $\mathbb{Z}[t,t^{-1}]$ and strictly simple quandles are actually simple modules isomorphic to $\mathbb{Z}[t,t^{-1}]/(g)$ where $g$ is an irreducible polynomial. 

The lattice of subquandles containing a fixed element a a strictly simple quandles is the two element lattice and so it has height $1$. We investigate the class of finite connected quandles for which this lattice has height at most $2$.

\begin{definition}
	A quandle $A$ is said to be {\it locally strictly simple} (LSS) if all its proper subquandles are strictly simple.
\end{definition}
%A quandle is LSS if the lattice of subquandles have height at most $2$. 
%Strictly simple quandles are examples of LSS quandles, but there exist non-trivial examples as non simple quandles of size 28 of the RIG database. As we shall see this class contains all (non-simple) connected quandles of size $pq$ and all (non-simple) connected quandles of size $4p$ with $p\geq 7$. 

%In this section we investigate the congruence lattice of connected LSS quandles.

Recall that connected quandles are congruence uniform (the blocks of a congruence have all the same cardinality) and that a quandle is called {\it subdirectly irreducible} if the intersection of all non-trivial congruence is non-trivial (and reducible otherwise). In particular if $\setof{\alpha_i}{i\in I}$ are trivially intersecting congruence of a quandle $Q$, then $Q$ embeds into $\prod_{i\in I} Q/\alpha_i$ \cite{UA}.  \\

\begin{lemma}\label{lemma_1_on_HSS}
Let $Q$ be a connected LSS quandle. Then:
\begin{itemize}
\item[(i)] $\alpha\wedge \beta=0_Q$ for every $\alpha,\beta\in Con(Q)$.
%\item[(ii)] Let $S$ and $S^{\prime}$ be subquandles of $Q$. Then $Q=Sg(S,S^{\prime})$.
\item[(ii)] $Q/\alpha$ is strictly simple for every $\alpha\in Con(Q)$.
\end{itemize}
Moreover, $Con(Q)$ is one of those in figure \ref{lattices2}.
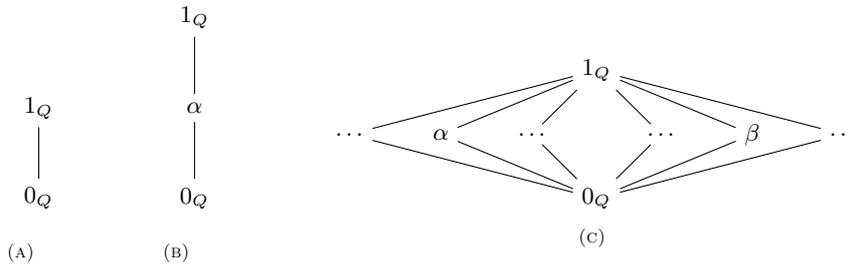
\begin{figure}[!h]
\subfloat[][]{
\begin{tikzpicture}[node distance=1.2cm]
\title{Lattices}
\node(A4)                           {$1_{Q}$};
\node(C32)      [below of=A4]       {$0_{Q}$};
\draw(A4)       -- (C32);
\end{tikzpicture}}\qquad\qquad
%\caption{Non subdirectly irreducible} }
\subfloat[][]{
\begin{tikzpicture}[node distance=1.2cm]
\title{Lattices}
\node(A4)                           {$1_{Q}$};
\node(V4)       [below of=A4] {$\alpha$};
\node(C32)      [below of=V4]       {$0_{Q}$};
\draw(A4)       -- (V4);
\draw(V4)       -- (C32);
\end{tikzpicture}}\qquad\qquad\subfloat[][]{
\begin{tikzpicture}[node distance=1.2cm]
\title{Lattices}
\node(A1)                           {$1_{Q}$};
\node(B)       [below right of=A1] {$\ldots$};
\node(BB)       [below left of=A1] {$\ldots$};
\node(B n-1)       [left of=BB] {$\alpha$};
\node(Bn)       [left of=B n-1] {$\ldots$};
\node(B 2)       [right of=B] {$\beta$};
\node(B 1)      [right of=B 2]  {$\ldots$};
\node(C)      [below right of=BB]       {$0_{Q}$};
\draw(A1)       -- (Bn);
\draw(A1)       -- (B n-1);
\draw(A1)       -- (B);
\draw(A1)       -- (BB);
\draw(A1)       -- (B 2);
\draw(A1)       -- (B 1);
\draw(Bn)       -- (C);
\draw(B n-1)       -- (C);
\draw(B)       -- (C);
\draw(BB)       -- (C);
\draw(B 1)       -- (C);
\draw(B 2)       -- (C);
\end{tikzpicture}}
\caption{(A) Simple - (B) Subdirectly Irreducible - (C) Subdirectly reducible.}\label{lattices2}
\end{figure}
\end{lemma}
\begin{proof}
(i) The blocks $[a]_\alpha$ and $[a]_\beta$ are strictly simple quandles. So either their intersection is trivial or they coincide. If $\alpha\neq \beta$, then $[a]_\alpha\bigcap [a]_\beta=\{a\}$ for some $a\in Q$. Since the blocks have all the same size then $\alpha\wedge\beta=0_Q$.

%(ii) Both $S,S^{\prime}$ are contained in $Sg(S,S^{\prime})$. Therefore, if $S\neq S^{\prime}$, $Sg(S,S^{\prime})=Q$.

(ii) Let $\alpha\in Con(Q)$. Then if $[a]_\alpha\neq [b]_\alpha$ then $[a]_\alpha$ is a proper subquandle of $Sg([a]_\alpha,[b]_\alpha)$ and so $Q=Sg([a]_\alpha,[b]_\alpha)$. Hence $Q/\alpha=	Sg([a]_\alpha,[b]_\alpha)$ for every $[a]_{\alpha}\neq [b]_{\alpha}$ and then $Q/\alpha$ is strictly simple.

The last statement follows by (i) and (ii).
\end{proof}

%Note that the congruence lattice in \ref{lattices2}(B) do not characterize LSS quandles (for instance there exists non-faithful connected quandles of size $p^3$ with congruence lattice as in \ref{lattices2}(B) and they are not LSS). 
A characterization of (non-simple) LSS quandles is offered by the following Proposition. 
\begin{remark}\label{on generators of LSS}
	Let $Q$ be a connected LSS quandle and $\alpha\in Con(Q)$. Since every block and the factor $Q/\alpha$ are generated by any pair of their elements then $Q$ is generated by any triple $\setof{a,b,c\in Q}{[a]_\alpha= [b]_\alpha \neq [c]_\alpha }$ (see also \cite[Lemma 6.1]{GB}).
\end{remark}
\begin{proposition}\label{LSS iff}
Let $Q$ be a non-simple connected quandle. The following are equivalent:
\begin{itemize}
\item[(i)] $Q$ is LSS.
\item[(ii)] $Q/\alpha$ and $[a]_\alpha$ are strictly simple quandles for every $\alpha\in Con(Q)$ and every $a\in Q$.
\item[(iii)] There exists $\alpha\in Con(Q)$, such that $Q/\alpha$ and $[a]_\alpha$ are strictly simple quandles for every $a\in Q$.
\end{itemize}
\end{proposition}
\begin{proof}
(i) $\Rightarrow$ (ii) It follows by Lemma \ref{lemma_1_on_HSS} (iii) and since $[a]_\alpha$ is a proper subquandle of $Q$.

(ii) $\Rightarrow$ (iii) Clear.

(iii) $\Rightarrow$ (i) Let $S$ be a subquandle of $Q$. Then either the image of $S$ under the canonical map $a\mapsto [a]_\alpha$ is the whole $Q/\alpha$ or it is just a one element set.\\
 In the first case, for every $[a]_\alpha$ in $Q/\alpha$	there exists $s_a\in S\cap [a]_\alpha$. If $S\cap [a]_\alpha=\{s_a\}$ for every $a\in Q$, then $S\cong Q/\alpha$ and so $S$ is strictly simple. If $ s_a\neq b\in S\cap [a]_\alpha1$ for some $a\in Q$, then $Sg(a,b,s_c)=Q\leq S$.\\
In the second case, $S\subseteq [a]_\alpha$ for some $a\in Q$, then either $|S|=1$ or $S=[a]_\alpha$, and so $S$ is strictly simple.
\end{proof}

The isomorphism classes of subquandles of $Q$ are completely determined by its factors and by the block of its congruences. Indeed, by the same argument used in Proposition \ref{LSS iff} to prove the implication (iii) $\Rightarrow$ (i) we have:

\begin{corollary}\label{subquandles}
Let $Q$ be a connected LSS non simple quandle. Every subquandle of $Q$ is isomorphic either to $[a]_\alpha$ or to $Q/\alpha$ for some $\alpha\in Con(Q)$.
\end{corollary}

%\begin{proof}
%Let $S$ be a subquandle of $Q$. Since $Q/\mu$ is stricltly simple for every $\mu\in Cong(Q)$, then the image of $S$ under the canonical mapping $a\mapsto [a]_\mu$ is either $\{[a]_\mu\}$ for some $a\in Q$, or $Q/\mu$. If there exists $a,b\in S$, such that $[a]_\mu=[b]_\mu$
%\end{proof}

\subsection*{Connected abelian LSS quandles}\label{direct prod}

In this section we describe finite, abelian connected LSS quandles. Finite connected abelian quandles are latin and they are polynomially equivalent to modules over $\mathbb{Z}[t,t^{-1}]$ \cite[Section 2.4]{Principal}. In particular finite connected affine quandles are latin and we can consider them as modules over $\mathbb{Z}[t]$. 
Indeed they admit an affine representation as $\aff(A,f)$ where $A$ is an abelian group and $f\in \aut{A}$. In particular congruences are in one-to-one correspondence with subquandles containing $0$ and with submodules \cite[Proposition 2.18, Remark 2.19]{Principal} i.e. with $f$-invariant subgroups. Moreover two connected finite affine quandles are isomorphic if and only if they are isomorphic as modules.
% We refer to \cite{modules} for definitions and results on modules. 

The first class of abelian LSS quandles is given by subdirectly reducible LSS quandles (i.e. quandle which congruence lattice is as in figure \ref{lattices2}(C) which are just direct product of a pair of strictly simple quandles.

\begin{theorem}\label{LSS reducible}
	Let $Q$ be a finite connected subdirectly reducible LSS quandle. Then $Q\cong M\times N$ where $M$ and $N$ are connected strictly simple quandles. 
\end{theorem}
\begin{proof}
	%Note that, for every $\alpha\in Con(Q)$, for every $[a]_\alpha \neq [b]_\alpha$, $Sg(a,b)\cong Q/\alpha$, or $Sg(a,b)=Q$. \\
	Lemma \ref{lemma_1_on_HSS} shows that any subdirectly reducible LSS connected quandle $Q$ has an embedding into the product of any pair of its factors, i.e., $$Q\hookrightarrow \, Q/\alpha\times Q/\beta,$$ 
	for every $\alpha,\beta \in Con(Q)$. 
	So, $Q$ embeds into a product of two strictly simple quandles. In particular $[a]_\alpha$ maps onto $Q/\beta$ and then $|Q|=|[a]_{\alpha}||Q/\alpha|\geq |Q/\beta||Q/\alpha|$. Therefore $Q/\cong Q/\alpha \times Q/\beta$.
	
\end{proof}

\begin{remark}
	Let $Q=M\times N$ where $M,N$ are strictly simple quandles. 
	\begin{itemize}
		\item[(i)] if $M\cong N$ then $Q$ is isomorphic to $\aff(\mathbb{F}_{|M|}^2,\lambda I)$ and its automorphism group is doubly transitive and any pair of elements generate a subquandle isomorphic to $M$ \cite[Section 3.3]{Principal}. In particular $Con(Q)$ has $|M|+1$ non trivial  congruences (it is the lattice of subspaces of a 2-dimensional vector space): 
		%By Maschke theorem, we have $Q\cong M^2$. Note that in this case $Q$ is doubly-homogeneous, according to \cite{Doubly}.\comment{check this}
		
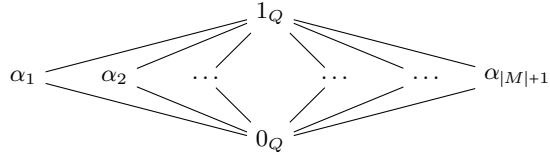
\begin{center}
			
			\begin{tikzpicture}[node distance=1.2cm]
			\title{Lattices}
			\node(A1)                           {$1_{Q}$};
			\node(B)       [below right of=A1] {$\ldots$};
			\node(BB)       [below left of=A1] {$\ldots$};
			\node(B n-1)       [left of=BB] {$\alpha_2$};
			\node(Bn)       [left of=B n-1] {$\alpha_1$};
			\node(B 2)       [right of=B] {$\ldots$};
			\node(B 1)      [right of=B 2]  {$\alpha_{|M|+1}$};
			\node(C)      [below right of=BB]       {$0_{Q}$};
			\draw(A1)       -- (Bn);
			\draw(A1)       -- (B n-1);
			\draw(A1)       -- (B);
			\draw(A1)       -- (BB);
			\draw(A1)       -- (B 2);
			\draw(A1)       -- (B 1);
			\draw(Bn)       -- (C);
			\draw(B n-1)       -- (C);
			\draw(B)       -- (C);
			\draw(BB)       -- (C);
			\draw(B 1)       -- (C);
			\draw(B 2)       -- (C);
			\end{tikzpicture}\, .
		\end{center}

		\item[(ii)] Otherwise, let $M=Q/\alpha$ and $N=Q/\beta$ and let $\gamma$ be a congruence different from $\alpha,\beta$. Then $M=[0]_{\beta}$ and $N=[0]_{\alpha}$ are both isomorphic to to $Q/\gamma$ since $\alpha\wedge \gamma=\beta\wedge\gamma=0_Q$. Hence $Q$ has just two proper congruences:
		
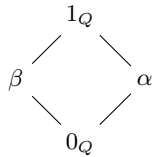
\begin{center}
			
			\begin{tikzpicture}[node distance=1.2cm]
			\title{Lattices}
			\node(A1)                           {$1_{Q}$};
			\node(B)       [below right of=A1] {$\alpha$};
			\node(BB)       [below left of=A1] {$\beta$};
			\node(C)      [below right of=BB]       {$0_{Q}$};
			\draw(A1)       -- (B);
			\draw(A1)       -- (BB);
			\draw(B)       -- (C);
			\draw(BB)       -- (C);
			\end{tikzpicture}\, .
		\end{center}
		\noindent Let $a,b\in Q$ such that $[a]_\alpha\neq [b]_\alpha$ and $[a]_\beta\neq [b]_\beta$. Then $Sg(a,b)$ maps onto both factors and then $Sg(a,b)= Q$.
	\end{itemize}
\end{remark}

Note that if $Q=\aff(A,f)$ is a subdirectly irreducible LSS connected quandle, namely the congruence lattice of $Q$ is as in \ref{lattices2}(B), then $A$ has a unique proper $f$-invariant subgroup.

\begin{lemma}\label{exp}
	Subdirectly irreducible abelian LSS connected quandle are affine over $p$-groups of exponent less or equal to $p^2$. 
\end{lemma}
\begin{proof}
	Let $Q=\aff(A,f)$. The $p$-components of $A$ provides trivially intersecting congruences as they are characteristic subgroups and then submodules. Moreover, the map $p:A\to A$, $a\mapsto pa$ provides a chain of submodules $ker(p)\leq ker (p^2)$ as they are characteristic subgroups. Therefore $A$ is a $p$-group with exponent less or equal than $p^2$.
\end{proof}
If the exponent of the group is $p$ a complete characterization is given by module theory.
\begin{proposition}\label{abelian and cyclic mod}
	Let $Q=\aff(\mathbb{Z}_p^n,f)$ be a connected quandle. The following are equivalent:
	\begin{itemize}
		\item[(i)] $Q$ is a subdirectly irreducible LSS quandle.
		\item[(ii)] $Q\cong \mathbb{Z}_p[t]/(g^2)$ as $\mathbb{Z}_p[t]$-module where $g$ is an irreducible polynomial over $\mathbb{Z}_p$. 
	\end{itemize}
	Moreover $Q$ is two-generated and $n$ is even.
\end{proposition}

\begin{proof}
	(i) $\Rightarrow$ (ii) Every affine quandle over $\mathbb{Z}_p^n$ is a module over $\mathbb{Z}_p[t,t^{-1}]$ and since $Q$ is finite we can consider it as a module over $\mathbb{Z}_p[t]$. The lattice of submodules of $Q$ is the three elements if and only if $Q$ is a primary and cyclic module \cite[Lemma 2]{modules2}.
%	modules over $\mathbb{Z}_p[t]$ admits a direct decomposition into primary components, and each primary components decomposes as a direct sum of cyclic submodules. Therefore $Q$ is a primary cyclic module. By \cite[Lemma 1]{modules} we can conclude that 
	According to \cite[Lemma 1]{modules}, as module $Q\cong \mathbb{Z}_p[t]/(g^{2})$ where $g$ is an irreducible polynomial.
	
	(ii) $\Rightarrow$ (i) The multiplication by $t+(g^2)$ in $\mathbb{Z}_p[t]/(g^2)$ is invertible. Indeed if $th\in (g^2)$ then $g^2$ divides $h$.	Let $Q=\aff(\mathbb{Z}_p[t]/(g^2),f)$ where $f(x+(g^2))=tx+(g^2)$. If $(1-f)(x+(g^2))=(1-t)x+(g^2)= (g^2)$. Then $g^2$ divides $(1-t)x$ and then $g^2$ divides $x$. Hence $R_{(g^2)}=1-f$ is invertible and then $Q$ is latin. Moreover $Q$ has just one submodule according to \cite[Lemma 1]{modules}. Then $Q$ is a subdirectly irreducible abelian LSS quandle since subquandles are cosets of submodules \cite[Remark 2.19]{Principal}.
	
	The last statement follows since $Q$ is a cyclic module and therefore $Q$ is generated as a quandle by $0$ and $a$ for some $a\in Q$. The dimension of $\dis(Q)$ is $2deg(p)$ \cite[Lemma 1]{modules}. 
\end{proof}
Note that quandles described in Proposition \ref{abelian and cyclic mod} are in one-to-one correspondence with irreducible polynomials as for connected strictly simple quandles. 

\begin{proposition}
	Let $Q=\aff(A,f)$ be a connected abelian subdirectly irreducible LSS quandle. If $exp(A)=p^2$, then $A\cong \mathbb{Z}_{p^2}^{n}$ and $Q/\alpha	\cong [a]$.
\end{proposition}
\begin{proof}
	Since $A$ has a unique $f$-invariant subroup, $Im(p)=ker(p)=pA$. Therefore $A\cong \mathbb{Z}_{p^2}^n$ and $Q/\alpha\cong  A/pA\cong pA=ker(p)\cong [a]$ as modules.
\end{proof}

\begin{example}
	Let $h=x^n+a_{n-1}x^{n-1}+\ldots+a_{1}x+a_{0}$ be an irreducible polynomial over $\mathbb{Z}_p$. Let  $$\setof{e_i=(\underbrace{0,\ldots, 1}_i,0,\ldots,0)}{1\leq i \leq n}$$ be a canonical set of generators for $A=\mathbb{Z}_{p^2}^n$ and  $Q=\aff(A,F)$ where
	\begin{center}
		$F=\begin{bmatrix}
		0 & 0 & \ldots &  \ldots &-a_{0} \\
		1& 0 &\ldots & \ldots & -a_{1} \\
		0 &1&\ldots &  \ldots &-a_{2} \\
		0 &0 &\ldots & \ldots & \ldots\\
		0 & 0 &\ldots &1 & -a_{n-1} \\
		\end{bmatrix}$.
	\end{center}
	The induced matrix on the factor $A/pA$ with respect to the basis $\setof{e_i+pA}{1\leq i \leq n}$ and the restriction to $pA$ with respect to the basis $\setof{pe_i}{1\leq i \leq n}$ are equal to $F$ which is the companion matrix of $h$. Therefore, $A/pA$ and $pA$ are isomorphic as simple modules and $Q$ is LSS and subdirectly irreducible by Proposition \ref{LSS iff}. In particular for every strictly simple quandle this construction provides a quandle with the desired properties. 
	%\comment{How many? Check on GAP make examples. Check Hou}
\end{example}

\section{Connected non-abelian LSS quandles}\label{non-ab LSS}

Connected non-abelian LSS quandles are subdirectly irreducible, since all quandles described in Theorem \ref{LSS reducible} are abelian. In the following we are describing such quandles according to different properties, in particular according to the equivalence relation $\sigma_Q$ (see \eqref{sigma}). So, we will denote by $\LSS(p^m,q^n,\sigma_Q)$ the class of finite non-abelian subdirectly irreducible connected quandles with given $\sigma_Q$ and such that the unique congruence has factor of size $q^n$ and blocks of size $p^m$.

The unique proper congruence of a quandle $Q\in \LSS(p^m,q^n,\sigma_Q)$ is $\gamma_Q$, since the unique factor is abelian.
\begin{corollary}\label{congruence of SI is [1,1]}
	Let $Q\in LSS(p^m,q^n,\sigma_Q)$ and let $G=\dis(Q)$. Then the unique non-trivial congruence of $Q$ is $\gamma_Q$,  $Q/\gamma_Q\cong \aff(G/\gamma_1(G),f_{\gamma_1(G)})$ and $\dis(Q)_a\leq\dis^{\gamma_Q}=\gamma_1(G)$ for every $a\in Q$.
\end{corollary}
\begin{proof}
The quandle $Q$ is not abelian and $Q/\alpha$ is abelian. Then $\gamma_Q$ is the unique non trivial congruence. Then we can apply \cite[Proposition 2.6]{GB}.
\end{proof}

%\comment{keep it?}
%If $Q$ is a subdirectly irreducible LSS quandle, the canonical map between $\dis(Q)$ and $\dis(Q/\gamma_Q)$ \eqref{pi alpha} can be extended to the automorphism group of $Q$. 
% 
%\begin{lemma}
%	Let $Q\in \LSS(p^m,q^n,\sigma_Q)$. If $Q/\gamma_Q$ and $[a]$ are not isomorphic, then the mapping
%	$$\psi:\aut{Q}\longrightarrow \aut{Q/\gamma_Q},\quad h\mapsto \psi(h):[a]\mapsto[h(a)]$$ is a well defined group homomorphism. In particular $\aut{Q}/\dis(Q)$ is cyclic.
%\end{lemma}
%\begin{proof}
%	
%According to Corollary \ref{subquandles} a subquandle $M$ of $Q$ is either a block of $\gamma_Q$ or its intersection with each block has size one and $M\cong Q/\gamma_Q$. Then the image of $[a]$ by any automorphism is a block of $\gamma_Q$ and so the map $\psi$ is well-defined. Moreover the composition
%$$\aut{Q}\overset{\psi}{\longrightarrow} \aut{Q/\gamma_Q}\longrightarrow \aut{Q/\gamma_Q}/\dis(Q/\gamma_Q)\cong \mathbb{Z}_{q^n-1}$$
%factorizes through the canonical projection onto $\aut{Q}/\dis(Q)$ and its image is cyclic (see \cite[Proposition 3.4]{Principal}).
%\end{proof}
%
%%We will denote by $\LSS(p^m,q^n,\sigma_Q)$ the class of subdirectly irreducible connected quandles with $|Q/\gamma_Q|=q^n$ and $|[a]|=p^m$. Note that if $p^m>2$ the $Q$ has no projection subquandle (the only strictly simple projection quandle has size $2$).
%
%%\subsection*{$\LSS(2,q^n)$}

\subsection*{The classes $\LSS(2,q^n,\sigma_Q)$}
Let assume that $Q$ is a finite LSS connected quandle with projection subquandles. Then $Q$ is $Q$ is not abelian, since it is not latin \cite[Corollary 2.6]{Principal} and it is subdirectly irreducible since all the quandles in Thereom \ref{LSS reducible} are latin. The factor $Q/\gamma_Q$ has no projection subquandles, hence the blocks of $\gamma_Q$ are projection subquandles of size $2$, since the only strictly simple projection quandle has size $2$. Therefore $Q\in \LSS(2,q^n,\sigma_Q)$.
% We will make use of the theory of degree 2 extension of affine quandles developed in \cite{Claw}

%\begin{proposition}\label{faithfulness}
%Let $Q$ be a connected SI LSS quandle. If $Q$ is non-faithful then $[a]\cong \mathcal{P}_2$.
%\end{proposition}
%\begin{proof}
%
%\end{proof}
%We need the following Lemma.

%\begin{lemma}\label{rank2}\comment{maybe not needed anymore}
%	Let $Q$ be a connected quandle and $\alpha\in Con(Q)$ such that $Q/\alpha\cong \aff(\mathbb{Z}_2^n,f)$ and $|[a]_\alpha|=2$. Then $Q$ is a principal non-faithful quandle. 
%\end{lemma}
%\begin{proof}
%	According to \cite[Proposition 6,7]{Claw} we have that $Q\cong (Q/\alpha\times \mathbb{Z}_2,\ast)$, where
%	\begin{displaymath}
%	(a,s)\ast (b,t)=(a\ast b,t+\beta(a,b)+\mu(a-b)s),
%	\end{displaymath}
%	and $\beta:Q/\alpha\times Q/\alpha\to \mathbb{Z}_2$ and $\mu:Q/\alpha\to \mathbb{Z}_2$ such that
%	\begin{equation}\label{prop of nu}
%	\mu (f(a))=\mu(a), \quad \mu(a+b)+\mu(a)\mu(b)=\mu(f(a)+b).
%	\end{equation}
%	Setting first $a=b$ and then $b=(1-f)(a)$ in \eqref{prop of nu} and using that $\mu(0)=2\mu(a)=0$ and $\mu(a)^2=\mu(a)$
%	%\begin{eqnarray*}
%	% \nu(a)&=&\nu((1+f)(a))\\  
%	%\nu(f(a))+\nu(a)\nu((1+f)(a))&=&\nu(a),
%	%\end{eqnarray*}
%	for every $a\in Q/\alpha$ we get that $\mu=0$. Therefore, $Q$ is not faithful and by \comment{cite covering} $Q$ is principal.
%\end{proof}

%Note that if $Q$ is a connected, LSS S.I. quandle, then $\mu=[1_Q,1_Q]$ since the factor is Abelian. In the following we will use the shortest notation $[1,1]$ for such congruence.
\begin{theorem}\label{LSS with P_2}
	Let $Q\in \LSS(2,q^n,\sigma_Q)$. Then either $|Q|=6$ or $q=2$ and $Q$ is a non-faithful principal quandle over an extraspecial $2$-group.
\end{theorem}

\begin{proof}
	Let	$|Q/\gamma_Q|=q^n$.
	If $Q$ is faithful then we can apply \cite[Theorem 4.8(i)]{Principal} and so $|Q|=6$.
	If $Q$ is not faithful then $\gamma_Q=\lambda_Q$ and $Q/\lambda_Q$ is affine, so $Q$ is principal over some group $G$ by \cite[Proposition 5.3]{GB}. Therefore, $|Q|=|G|=2q^n$. If $q\geq 3$ then $G$ has a normal $p$-Sylow group of index $2$. According to \cite[Theorem 2.10]{Principal} it induces a congruence with non-connected factor of size $2$, contradiction.\\
	If $q=2$ then $\gamma_Q=\zeta_Q$, since any $p$-quandle have non-trivial center by \cite[Corollary 6.6]{CP} and $Q$ is non-abelian. Moreover $G=\dis(Q)$ has just one characteristic subgroup, namely $\dis^{\gamma_1(Q)}=\gamma_1(G)=Z(G)=\Phi(G)\cong \mathbb{Z}_2$. So $G$ is an extraspecial $2$-group.
\end{proof}
Note that there are $2$ non-isomorphic connected quandles of size $6$ in the RIG library of GAP and they are both faithful, subdirectly irreducible and LSS, in particular they are in $\LSS(2,3,\gamma_Q)$. Using the criterion of connectedness for quandles represented by nilpotent groups given in \cite[Corollary 2.8]{GB} we can identify the automorphisms of extraspecial $2$-groups giving rise to quandles $\mathcal{Q}(G,f)$ in $\LSS(2,2^{2n},1_Q)$.
\begin{proposition}\label{2 extraspecial}
	Let $G$ be an extraspecial $2$-group of size $2^{2n+1}$ and $f\in \aut{G}$. The following are equivalent:
	\begin{itemize}
		\item[(i)] $Q=\mathcal{Q}(G,f)\in\LSS(2,2^{2n},1_Q)$.
		\item[(ii)] $f_{\gamma_1(G)}$ acts irreducibly on $G/\gamma_1(G)$.
	\end{itemize}
	Isomorphism classes of such quandles are in one-to-one correspondece with conjugacy classes of such automorphisms.
\end{proposition}
\begin{proof}
(i) $\Rightarrow$ (ii) It follows since the factor $Q/\gamma_Q$ is strictly simple.

(ii) $\Rightarrow$ (i) The quandle $\aff(G/\Phi(G),f_{\Phi(G)})$ is connected. According to \cite[Proposition 2.7]{GB}, $Q$ is connected and $\dis(Q)\cong G$. Moreover $\Phi(G)=\gamma_1(G)=Z(G)\cong \mathbb{Z}_2$ is contained in every $N\in Norm(Q)$ since $G$ is nilpotent, so $Q$ is subdirectly irreducible. The factor with respect to the congruence defined by $Z(G)$ is simple, therefore $Q$ is LSS by Proposition \ref{LSS iff}.  

Isomorphism classes of connected principal quandles over a group $G$ correspond to conjugacy classes of automorphisms (Theorem \ref{iso theorem1}).
\end{proof}

\begin{example}
	
A GAP computer search for quandles as in Corollary \ref{2 extraspecial} reveals that there exist just one isomorphism class of such quandles as of size $8,32,128$ and two of size $512$. All these examples are over the extraspecial $2$-group constructed as the central product of one copy of the quaternion groups and several copies of the dihedral group. %Extraspecialgroup($2^{2n+1}$,"-"). 

\end{example}
\subsection*{The classes $\LSS(p^m,q^n,\sigma_Q)$}
From now on we assume that $p^m>2$ and so the quandles in $\LSS(p^m,p^n,\sigma_Q)$ have no projection subquandles. In particular they are faithful (the quandles in Proposition \ref{2 extraspecial} are the only non-faithful connected LSS quandles). 
%\begin{lemma}\label{LSS then faithful}
%	Every quandle in $\LSS(p^m,q^n,\sigma_Q)$ is faithful.
%\end{lemma}
%\begin{proof}
%	(i) Let $L_a=L_b$, then $L_{[a]}=L_{[b]}$, therefore $[a]=[b]$, since $Q/\gamma_Q$ is latin. Moreover, since $L_{a}|_{[a]}=L_b|_{[a]}$, then $a=b$ since $[a]$ is latin. 
%\end{proof}

For faithful quandles the relation between the congruence lattice and the lattice $Norm(Q)$ has better properties. %For instance, and $N\in Norm(Q)$ then \comment{$\c{N}=0_Q$ if and only if $N=1$, needed?} 
Indeed if $Q$ is faithful, a congruence $\alpha$ is abelian (central) if and only if $\dis_\alpha$ is abelian (central) and $\zeta_Q=\c{Z(\dis(Q))}$. So for quandles in $\LSS(p^m,q^n,\sigma_Q)$ the subgroup $\gamma_1(\dis(Q))$ is the biggest element of $Norm(Q)$ (in particular it is a maximal characteristic subgroup) and $\dis_{\gamma_Q}$ is the smallest. Moreover finite connected faithful nilpotent quandles splits as direct product of quandles of prime power size \cite[Theorem 1.4]{CP}, hence the subdirectly irreducible ones have size a power of a prime.

 \begin{lemma}\label{embedding of ker}
 	Let $Q\in \LSS(p^m,q^n,\sigma_Q)$ and let $G=\dis(Q)$. Then:
 	\begin{itemize}
 	
 		\item[(i)] $\gamma_2(G)\leq \dis_{\gamma_Q}\cong \mathbb{Z}_p^{m+k}$ where $k\leq m$ and $\gamma_Q$ is abelian.
 		 
 		\item[(ii)] If $p\neq q$ then $\gamma_2(G)=\dis_{\gamma_Q}$ and $Z(G)=1$.
 	\end{itemize}
 \end{lemma}
 \begin{proof}
 	(i) Since $Q=Sg(a,[b])$ then $\dis^{\gamma_Q}$ embeds into $\aut{[a]}\times \aut{[b]}\cong (\mathbb{Z}_p^m\rtimes \mathbb{Z}_{p^m-1})^2$ . Hence $\dis_{\gamma_Q}\leq \dis^{\gamma_Q}$ is solvable and then it is abelian by Lemma \ref{minimal congruences}. Moreover $\gamma_Q=\mathcal{O}_{\dis_{\gamma_Q}}$ since the blocks are connected. We can apply Lemma \ref{embedding as quandle}(ii) to $\dis_{\gamma_Q}$, so it is isomorphic to $\mathbb{Z}_p^{m+k}$ with $k\leq m$ since it is transitive on each block. Then $\gamma_Q$ is abelian according to \cite[Theorem 1.1]{CP} and $\gamma_2(G)\leq \dis_{\gamma_Q}$ by \cite[Proposition 3.3]{CP}.

 	(ii) If $p\neq q$ then $Q$ is not nilpotent, and  then $\gamma_2(G)=\dis_{\gamma_Q}$ since $\dis_{\gamma_Q}$ is minimal in $Norm(Q)$. Moreover $\zeta_Q=\c{Z(G)}=0_Q$ and accordingly $Z(G)= 1$. 
 	%\end{proof}
 	%Remember that $\gamma_2(G)/\gamma_1(G)$ is generated by $\binom{n}{2}=\frac{n(n-1)}{2}$ where $n$ is the number of generators of $G/\gamma_1(G)$ \comment{cite}. 
 \end{proof}
  In the sequel we discuss several cases according to the equivalence relation $\sigma_Q$. We can have the following cases:

\begin{lemma} \label{cases for sigma}
	%\comment{Do the blocks of $\sigma_Q\cap [1,1]$ have all the same size? I think so, because the intersection of two blocks is a a block.}
	Let $Q \in  \LSS(p^m,q^n,\sigma_Q)$. Then one of the following holds:
	\begin{itemize}
		\item[(i)] $\sigma_Q=1_Q$ and $Q$ is principal.
		\item[(ii)] $\sigma_Q=\gamma_Q$.
		%\item[(iii)] $\sigma_Q=0_Q$;
		\item[(iii)] $\sigma_Q\wedge \gamma_Q = 0_Q$.
	\end{itemize}
	In particular if $p=q$ then $\gamma_Q\leq \sigma_Q$.
\end{lemma}
\begin{proof}
	The blocks of $\sigma_Q$ are subquandles of $Q$ and they are blocks with respect to the action of $\lmlt(Q)$. Assume that $\sigma_Q\neq 1_Q$. All the subquandles of $Q$ are strictly simple, so either $[a]_{\gamma_Q}$ and $ [a]_{\sigma_Q}$ coincide or their intersection is the singleton $\{a\}$. \\
	%Since $[a]\bigcap [a]_{\sigma_Q}$ is a block with respect to the action of $\lmlt(Q)$, then the classes of $\gamma_1(Q)\bigcap \sigma_Q$ have all the same size. \\
	If $p=q$, the quandle $Q$ is nilpotent and then $\gamma_Q$ is a central congruence, i.e. $\gamma_Q= \zeta_Q\leq\sigma_Q$. 
\end{proof}

 %So the cases in Lemma \ref{cases for sigma} correspond to (i) $\dis(Q)_a=1$, (ii)  
%\begin{proof}
%The factor $Q/\alpha$ is a connected finite strictly simple quandle and so it is Abelian. Hence $0_Q<[1_Q,1_Q]\leq \alpha$, i.e. $\alpha=[1_Q,1_Q]$. The last statement is \cite[]{GB}.
%\end{proof}

Recall that, if $Q$ is connected, the block of $a$ with respect to $\sigma_Q$ is the orbit of $a$ under the action of $N(\dis(Q)_a)$ \cite[Theorem 2.4]{Principal}.

\begin{lemma}\label{mu smaller than S}
	Let $Q\in \LSS(p^m,q^n,\sigma_Q)$, $p\neq q$ and $G=\dis(Q)$. The following are equivalent:
	\begin{itemize}
		\item[(i)] $\gamma_Q\leq \sigma_Q$.
		\item[(ii)] $\gamma_1(G)=\gamma_2(G)\cong \mathbb{Z}_p^{m+k}$ with $ k\leq m$.
		%\item[(iii)] $\dis^\mu(Q) \leq N(\dis(Q)_a)$ for every $a\in Q$.
		\item[(iii)] $Z(\gamma_1(G))\neq 1$.
	\end{itemize}
	
\end{lemma}
\begin{proof}
	(i) $\Rightarrow$ (ii) The subgroup $\gamma_1(G)$ is $\gamma_Q$-semiregular and transitive on each block and $Q/\gamma_Q$ is generated by two elements. By Lemma \ref{embedding as quandle}(ii) $\gamma_1(G)$ embeds into $\mathbb{Z}_p^{2m}$ and has at least size $p^m$. According to \cite[Proposition 4.7]{Principal} $G/\gamma_2(Q)$ is a $q$-group and since $|G|=|\gamma_1(G)|q^n=p^{m+k}q^n$ then $\gamma_1(G)=\gamma_2(G)$.
	
	%(ii) $\Leftrightarrow$ (iii) The blocks of $S$ are given by $[a]_S=a^{N(\dis(Q)_a)}$. So $[a]_\mu\leq \leq [a]_S$ if and only if $\dis^\mu(Q)\leq N(\dis(Q)_a)$ for every $a\in Q$.
	
	(ii) $\Rightarrow$ (iii) Clear. 
	
	(iii) $\Rightarrow$ (i) Let $Z=Z(\gamma_1(G))$. The orbits of $Z$ are contained in the blocks of $\sigma_Q$: indeed $\dis(Q)_a\leq \gamma_1(G)$ and then $Z$ centralizes $\dis(Q)_a$. Then $a\neq a^{Z}\subseteq  [a]_{\sigma_Q}\cap [a]_{\gamma_Q}$ and so $[a]_{\gamma_Q}\leq [a]_{\sigma_Q}$, since $[a]_{\gamma_Q}$ is strictly simple for every $a\in Q$. Hence $\gamma_Q\leq \sigma_Q$.
\end{proof}

\subsubsection*{Case $p=q$:} let us start with the case $p=q$. In this case $\gamma_Q=\zeta_Q\leq\sigma_Q$ and so $Q$ is nilpotent of length $2$. We can characterize the displacement group of such quandles using the following Proposition.

\begin{proposition}\label{sims} \cite[Proposition 9.2.5]{Sims}
	Let $G$ be a group and let $X\subseteq G$ such that $G/\gamma_1(G)=\langle\setof{ x\gamma_1(G)}{x\in X}\rangle$. Then $\gamma_1(G)/\gamma_2(G)=\langle \setof{[x,y]\gamma_2(G)}{x,y\in X}$. In particular if $G/\gamma_1(G)$ is cyclic, then $\gamma_1(G)=\gamma_2(G)$.
\end{proposition}

\begin{theorem}\label{p=q case}
Let $Q\in \LSS(p^m,p^n,\sigma_Q)$, $G=\dis(Q)$ and $a\in Q$. Then $Q$ is latin, $Z(G)=\dis_{\zeta_Q}\cong \mathbb{Z}_p^m$ and
$$ \Phi(G)=\gamma_1(G) \cong Z(G) \times G_a\cong \mathbb{Z}_p^{m+k}.$$ In particular $n\geq 2$.% and $k\leq min\{m,\binom{n}{2}\}$.
\end{theorem}
\begin{proof}
According to \cite[Corollary 5.2]{GiuThe} $G$ is a $p$-group. The displacement group of the factor $G/ \gamma_1(G)$ is elementary abelian, so $ \gamma_1(G)= \Phi(G)$. The quandle $Q$ is a finite connected faithful quandle of nilpotency length $2$. Then according to \cite[Proposition 3.5]{GB} $Q$ is latin, $Z(G)=\dis_{\zeta_Q}\cong \dis([a]) \cong \mathbb{Z}_p^m$, the kernel $\gamma_1(G)$ decomposes as the direct product of the center and the stabilizer $G_a$ and the last isomorphism follows from Lemma \ref{mu smaller than S}(ii). According to Proposition \ref{sims}, $G/\gamma_1(G)$ can not be cyclic, therefore $n\geq 2$. %and by Proposition \ref{sims}, $\gamma_1(G)/\gamma_2(G)\cong \mathbb{Z}_p^k$ has at most $\binom{n}{2}$ generators, therefore $k\leq \binom{n}{2}$. 
%If $x_1\gamma_1(G),\ldots, x_n\gamma_1(G)$ is a basis of $G/\gamma_1(G)$ then by \cite[]{Sims} $\gamma_1(G)/ \gamma_2(G)$ is generated by $\setof{[x_i,x_j]}{1\leq i,j,\leq n }$, so $k\leq \binom{n}{2}$ generators. 
%By Lemma \ref{congruence of SI is [1,1]} $\dis^\zeta_Q=\gamma_1(G)$ and it equals the frattini subgroup since $G/\gamma_1(G)$ is an elementary p-group and $\gamma_1(G)\leq \Phi(G)$. 
%The equality $\dis_{\zeta_Q}=Z(G)$ and structure of $\gamma_1(G)$ follows by \cite[]{GB}. The isomorphism in item (ii) follows from Lemma \ref{mu smaller than S}, since $\zeta_Q\leq \sigma_Q$. Since $Q$ is a central extension of a latin quandle and $Q$ is faithful, then $Q$ is latin. By Proposition \ref{sims} we have that $\gamma_1(G)/\gamma_2(G)\cong \mathbb{Z}_p^k$ has at most dimension ${n}\choose {2}$. 
\end{proof}

The class $\LSS(p^m,p^n,1_Q)$ is given by a class of principal quandles over special $p$-groups.
\begin{proposition}
Let $Q$ be a finite quandle. The following are equivalent:
\begin{itemize}
\item[(i)] $Q$ is principal latin quandle in $\LSS(p^m,p^n,1_Q)$.
%\item[(ii)] $Q\in \LSS(p^m,p^n,\sigma_Q)$ and $\gamma_2(G)=1$.

\item[(ii)] $Q\cong \mathcal{Q}(G,f)$ where $G$ is a special $p$-group, $f$ acts irreducibly on $Z(G)\cong \mathbb{Z}_p^m$ and $f_{Z(G)}$ acts irreducibly on $G/Z(G)\cong \mathbb{Z}_p^n$.
\end{itemize}
In particular $m\leq \frac{n(n-1)}{2}$.\\
Isomorphism classes of such quandles are in one-to-one correspondece with conjugacy classes of such automorphisms. 
\end{proposition}

\begin{proof}
(i) $\Rightarrow$ (ii) Let $Q\in \LSS(p^m,p^n,1_Q)$ and let $G=\dis(Q)$. Then $\gamma_2(G)\leq \dis_{\zeta_Q}=Z(G)=\gamma_1(G)$ and then $\gamma_2(G)=1$ and so $G$ is a special $p$-group. By Lemma \ref{rep_for_factors}, $Q/\zeta_Q\cong \aff(G/Z(G),f_{Z(G)})$. Since $Q$ is LSS then the action of $f$ over $Z(G)=[1]_{\gamma_Q}$ and the action of $f_{Z(G)}$ over $\dis(Q/\zeta_Q)\cong G/Z(G)$ is irreducible.

(ii) $\Rightarrow$ (i) The coset partition with respect to $Z(G)$ is a congruence of the quandle $Q=\Q(G,f)$ with strictly simple factor and strictly simple blocks. Then $Q$ is LSS according to Proposition \ref{LSS iff}.

	By Proposition \ref{sims}, $\gamma_1(G)\cong \mathbb{Z}_p^m$ has at most $\binom{n}{2}$ generators, therefore $m\leq \binom{n}{2}=\frac{n(n-1)}{2}$.
\end{proof}
%\begin{example}
%	Let $Q\in \LSS(p^m,p^n,\sigma_Q)$ and let $G=\dis(Q)$. If $\sigma_Q=1_Q$ then $\gamma_2(G)\leq \dis_{\zeta_Q}=Z(G)=\gamma_1(G)$ and then $\gamma_2(G)=1$. On the other hand if $\gamma_2(G)=1$ then $Q$ is principal \cite[Corollary 2.4]{GB}. Moreover, by Proposition \ref{sims}, $\gamma_1(G)\cong \mathbb{Z}_p^m$ has at most $\binom{n}{2}$ generators, therefore $m\leq \binom{n}{2}$. \\
%	If $\sigma_Q=\zeta_Q$ then $\gamma_2(G)=Z(G)$ and $\gamma_1(G)/\gamma_2(G)\cong \mathbb{Z}_p^k$ and according to Proposition \ref{sims} $k\leq \binom{n}{2}$.
%\end{example}

\begin{example}
	Let $G$ be and extraspecial $p$-group and let $Q=\mathcal{Q}(G,f)$ be a LSS subdirectly irreducible connected quandle. According to \cite[Corollary 1]{winter1972} the exponent of $G$ is $p$. According to the results of \cite{GB}, there are $\frac{(p-1)^2}{2}$ non-isomorphic quandles in $\LSS(p,p^2,1_Q)$ and such quandles are the unique non-abelian subdirectly irreducible connected LSS quandles of size $p^3$.
\end{example}

\subsubsection*{\textit Case $p \neq q$}: now we consider the case $p\neq q$. First we can characterize the structure of the displacement group. The subgroup $\gamma_1(\dis(Q))$ is a maximal characteristic subgroup of $\dis(Q)$, the quotient $\dis(Q)/\gamma_1(G)$ and $\gamma_2(\dis(Q))$ are elementary abelian group. Then we can apply \cite[Proposition 4.7]{Principal}.
\begin{theorem}\label{mu p neq q}
Let $Q\in  \LSS(p^m,q^n,\sigma_Q)$ and $p\neq q$. If $\gamma_Q\leq \sigma_Q$ then 
\begin{equation}\label{dis for gamma=sigma}
\dis(Q)\cong \mathbb{Z}_p^{m+k}\rtimes_\rho \mathbb{Z}_q^n
\end{equation}
 where $k\leq m$, $\rho$ is a faithful action and $\gamma_1(\dis(Q))=\gamma_2(\dis(Q))=\dis_{\gamma_Q}$.
\end{theorem}
\begin{proof}
	By virtue of Lemma \ref{embedding of ker} and Lemma \ref{mu smaller than S} we can apply \cite[Proposition 4.7(i)]{Principal} and then $\dis(Q)\cong \mathbb{Z}_p^{m+k}\rtimes_\rho \mathbb{Z}_q^n$. The kernel of $\rho$ is a subgroup of the center of $G$. According to Lemma \ref{embedding of ker} $Z(G)=1$ and then the action $\rho$ is faithful. %and $Fix(\rho)=1$.
\end{proof}

Let us consider the case $\sigma_Q\wedge \gamma_Q=0_Q$.
% By virtue of Lemma \ref{mu smaller than S} then $\dis_{\gamma_Q}=\gamma_2(\dis(Q))\neq \gamma_1(\dis(Q))$.
\begin{theorem}\label{dis mu and S =0}
Let $Q\in \LSS(p^m,q^n,\sigma_Q)$ and $\sigma_Q\wedge \gamma_Q=0_Q$. Then 
\begin{equation}\label{dis with extraspecial}
%\gamma_1(G) &\cong & \dis_{\gamma_Q} \rtimes \mathbb{Z}_q\cong \mathbb{Z}_{p}^{m+k}\rtimes_{\sigma} \mathbb{Z}_q,\\
\dis(Q)\cong \mathbb{Z}_p^{m+k}\rtimes_{\rho} K,
\end{equation}
 where $K$ is an extraspecial $q$-group, $q$ divides $p^m-1$, $n$ is even and $k\leq m$. If $q$ is odd then $\exp(K)=q$.
 % and $Fix(\sigma)=\{0\}$.
\end{theorem}
\begin{proof}
Let $G=\dis(Q)$. By virtue of Lemma \ref{embedding of ker} and Lemma \ref{mu smaller than S} we can apply \cite[Proposition 4.7(ii)]{Principal} and then $G\cong \gamma_2(G)\rtimes_{\rho} K$ where $K$ is a special $q$-group and $Z(K)\cong \gamma_1(G)/\gamma_2(G)$.

The automorphism group of a block is $ \aut{[a]}= \mathbb{Z}_p^m \rtimes  \mathbb{Z}_{p^m-1}$ and the only fixed-point-free automorphisms have order $p$. The elements $z\in Z(K)$ have order $q$  then $z\in \dis(Q)_{c_a}$ for some $c_a\in [a]$ for every $[a]\in Q/\gamma_Q$. Since $Q=Sg(a,[b])$, then
$$Z(K)\longrightarrow \aut{[b]}_{c_b}\cong \mathbb{Z}_{p^m-1},\quad h\mapsto h|_{[b]}$$
is an embedding and therefore $Z(K)$ is cyclic and $q$ divides $p^m-1$. Thus $K$ is a extraspecial $q$-group, so $n$ is even.

The induced automorphism over $K$ acts irreducibly on $K/Z(K)$. According to \cite[Corollary 1]{winter1972} if $q$ is odd then $K$ has exponent $q$.
% The last statement follows since $Z(\dis^{[1,1]})=\{1\}$. 
%
%
%
%The $\dis_\mu$ is semiregural and so by \comment{cite me and G} $\dis^\mu\cong \dis_\mu \rtimes \dis(Q)_a$. Let
%
%The embedding defined in Lemma \ref{generalities on LSS SI} (ii) factor through $\dis_\mu$ is this way:
%\begin{displaymath}
%    \xymatrixcolsep{63pt}\xymatrixrowsep{30pt}\xymatrix{
%    \dis^\mu \ar[r]^{\psi}
%    \ar[d] & \left(\mathbb{Z}_p^{m}\rtimes \mathbb{Z}_{p^m-1}\right)^2\ar[d] \\
%    \dis^\mu/\dis_\mu \ar[r]& \left(\mathbb{Z}_{p^m-1}\right)^2,}
%\end{displaymath}
%    since $\psi(\dis_\mu)$
\end{proof}

Note that the action $\rho$ in \eqref{dis for gamma=sigma} and in \eqref{dis with extraspecial} can be extended to the whole displacement group setting $\rho(x)=1$ whenever $x\in \gamma_2(\dis(Q))$. Sometimes we use such extension of $\rho$ instead of $\rho$ itself.\\
The equivalence $\sigma_Q$ is controlled by the properties of the action of $\dis_{\gamma_Q}$:
\begin{lemma}\label{sigma_Q=0 iff}
Let $Q\in  \LSS(p^m,q^n,\sigma_Q)$ and $\gamma_Q\wedge \sigma_Q=0_Q$. The following are equivalent:
\begin{itemize}
\item[(i)] $\sigma_Q = 0_Q$.
\item[(ii)] $\left(\dis_{\gamma_Q}\right)_a\neq 1$ for every $a\in Q$.
\end{itemize}

\end{lemma}

\begin{proof}
Let $\sigma_Q\neq 0_Q$ and $b\, \sigma_Q\, a$ such that $b\notin [a]_{\gamma_Q}$. Therefore, $\left(\dis_{\gamma_Q} \right)_a=\left(\dis_{\gamma_Q} \right)_b$ and $\left(\dis_{\gamma_Q} \right)_a$ fixes $[a]$ and $b$ which generate $Q$. Hence $\left(\dis_{\gamma_Q} \right)_a=1$. On the other hand if $\left(\dis_{\gamma_Q} \right)_a=1$ then
 $\dis(Q)_a\cong \mathbb{Z}_q$ for every $a\in Q$ and it embeds into $\aut{[b]}\cong \mathbb{Z}_p^m\rtimes \mathbb{Z}_{p^m-1}$ for every $b\notin [a]$. The fixed-point-free elements of $\aut{[b]}$ have order $p$, therefore $\dis(Q)_a$ fixes some $c_b\in [b]$ and so $\dis(Q)_a=\dis(Q)_{c_b}$, i.e. $\sigma_Q\neq 0_Q$.
\end{proof}
\begin{corollary}\label{faithful action and sigma =0}
	Let $Q\in\LSS(p^m,q^n,0_Q)$. Then the action $\rho$ in \eqref{dis with extraspecial} is faithful and $N\left(\left(\dis_{\gamma_Q}\right)_a\right)=\dis^{\gamma_Q}$.
\end{corollary}
 \begin{proof}
  Since $\sigma_Q=0_Q$ then $N(\dis(Q)_a)=\dis(Q)_a$ for every $a\in Q$ and $\left(\dis_{\gamma_Q}\right)_a$ is not trivial by Lemma \ref{sigma_Q=0 iff}.\\
   If $\rho(h)=1$ then $h\in N(\dis(Q)_a)\bigcap K=\dis(Q)_a\bigcap K\leq \gamma_1(\dis(Q))\bigcap K\leq Z(K)$. So $h\in Z(\dis(Q))=1$ and then $\rho$ is faithful.\\ 
  The congruence $\gamma_Q$ is abelian, so $\left(\dis_{\gamma_Q}\right)_{a}=\left(\dis_{\gamma_Q}\right)_{b}$ whenever $a\,\alpha\,b$, and then $\dis^{\gamma_Q}$ normalizes $\left(\dis_{\gamma_Q}\right)_{a}$. On the other hand the subgroup $\left(\dis_{\gamma_Q}\right)_{a,b}$ is trivial whenever $b\notin[a]$ since $Q$ is generated by $[a]$ and   $b$. So, if $h$ normalizes $\left(\dis_{\gamma_Q}\right)_a$ then $\left(\dis_{\gamma_Q}\right)_{a}=\left(\dis_{\gamma_Q}\right)_{h(a)}=\left(\dis_{\gamma_Q}\right)_{a,h(a)}$, and so $h(a)\,\alpha\, a$ i.e. $h\in \dis^{\gamma_Q}$. 
\end{proof}
%\comment{Is it true that if $\sigma_Q\wedge\gamma_Q=0_Q$ then $\sigma_Q=0_Q$?}
%
Note that if $\sigma_Q\neq 0_Q$ then $\dis_{\gamma_Q}\cong \mathbb{Z}_p^m$ and so it has the minimal admitted size. On the other extreme, i.e. when the size of $\dis_{\gamma_Q}$ is maximal, we have the following:
\begin{proposition}\label{prop of LSS}
Let $Q\in  \LSS(p^m,q^n,0_Q)$. If $\dis_{\gamma_Q}\cong \mathbb{Z}_p^{2m}$ then $Sg(a,b)\cong Q/\gamma_Q$ whenever $[a]\neq [b]$ and $Q$ is latin. 
%\begin{itemize}
%\item[(i)] $Sg(a,b)\cong M$ with $M\in \{Q,Q/\mu\}$ for every $[a]\neq [b]$. If $M= Q/\mu$ then $Q$ is latin.
%\item[(ii)] $|a^{L_b}|=|f|$ if $[a]=[b]$ and $|a^{L_b}|=|g|$ otherwise.
\end{proposition}
\begin{proof}
Note that $Sg(a,b)\cong h(Sg(a,b))=Sg(a,h(b))$ for every $a,b\in Q$ and every $h\in \dis(Q)_a$. The action of $\left(\dis_{\gamma_Q}\right)_a\cong \mathbb{Z}_p^m$ restricted to $[b]$ is regular whenever $[a]\neq [b]$, since $Q=Sg([a],c)$ for every $c\in [b]$ and $\left(\dis_{\gamma_Q}\right)_a$ fix all the elements of $[a]$. If $Q=Sg(a,b)$ for some $a,b\in Q$, then $Q=Sg(c,d)$ for any $c\,\gamma_Q \, a$ and $d\,\gamma_Q\,b$. Any element of order $q$ in $\gamma_1(\dis(Q))$ fixes one element in each block of $\gamma_Q$ and then it fixes some pair of generators, contradiction. Let $b*a=c*a\in Sg(a,b)$. Since $Sg(a,b)$ is isomorphic to $[a]$ or to $Q/\gamma_Q$, it is latin. Therefore $c=(b*a)/a=b$ and $Q$ is latin. 
\end{proof}

\begin{remark}
	Let $Q\in  \LSS(p^m,q^n,\sigma_Q)$ and $G=\dis(Q)$. Then the Frattini subgroup $\Phi(G)$ is properly contained in the Fitting subgroup $F(G)$ ($G$ is solvable) and $\gamma_2(G)\leq F(G)\leq \gamma_1(G)$ since $F(G)$ is characteristic and $\gamma_2(G)$ is abelian. If $\gamma_1(G)=\gamma_2(G)$ then $F(G)=\gamma_2(G)$ and $\Phi(G)=1$. Otherwise, $\gamma_2(G)$ has prime index in $\gamma_1(G)$ and $\gamma_1(G)$ has trivial center. So again $F(G)=\gamma_2(G)$ and $\Phi(G)=1$.
\end{remark}

\subsection*{The classes $\LSS(p,q^n,\sigma_Q)$}
In this section we turn our attention to the case $m=1$.

\begin{corollary}\label{cor for pq}
Let $Q\in \LSS(p^m,q^n,\sigma_Q)$ and $p\neq q$.
\begin{itemize}
\item[(i)] If $n=1$ then $\gamma_Q \leq \sigma_Q$.% and $\gamma_1(\dis(Q))=\gamma_2(\dis(Q))$.
\item[(ii)] If $m=1$ then $\dis_{\gamma_Q}\cong \mathbb{Z}_p^2$.
\end{itemize}
In particular the class $\LSS(p,q^n,1_Q)$ is empty. 
\end{corollary}
\begin{proof}
(i) Let $G=\dis(Q)$. According to Proposition \ref{sims}, $\gamma_1(G)=\gamma_2(G)$ since $\dis(Q/\gamma_Q)\cong G/\gamma_1(G)$ is cyclic.  So by Lemma \ref{mu smaller than S}, $\gamma_Q\leq \sigma_Q$.

(ii) The congruence $\zeta_Q$ is trivial, hence $\gamma_Q$ is not central. The subgroup $\dis_{\gamma_Q}$ embeds into $\mathbb{Z}_p^2$ and according to \cite[Lemma 4.4]{Principal} it is not cyclic. So $\dis_{\gamma_Q}\cong \mathbb{Z}_p^2$. In particular $Q$ is not principal, since $\dis(Q)_a$ is non-trivial.
%The group $\mathbb{Z}_q^n$ embeds into $\mathbb{Z}_p^{1+k}$ by virtue of \ref{embedding of ker} with $k=0,1$. If $p\neq  q$, then necessarily $k=1$, otherwise $\mu$ would be central. Since the blocks of ${[1,1]}$ have prime size and the blocks of $\sigma_Q$ are contained in ${[1,1]}$, then they either are equal or $\sigma_Q=0_Q$. \comment{check!!!}
\end{proof}
\noindent Let us consider the case $\gamma_Q=\sigma_Q$.

\begin{proposition}\label{ex on 2to the n p}
Let $Q\in \LSS(p,q^n,\gamma_Q)$. Then $q>2$ and $n=1$.
\end{proposition}

\begin{proof}
By virtue of Corollary \ref{cor for pq}(ii) $Q$ is not principal and by Theorem \ref{mu p neq q} $\dis(Q)=G\cong \mathbb{Z}_p^2\rtimes_{\rho} \mathbb{Z}_q^n$ where $\rho$ is a faithful action. Then $q$ divides $p^2-1$. \\
If $q$ divides $p-1$ or $q=2$ then the action of $\mathbb{Z}_q^n$ is diagonalizable (see Lemma \ref{eigenvalues of A}) and then $\mathbb{Z}_q^n$ embeds into the subgroup of diagonal matrices of order $q$ (see Lemma \ref{normalizer of matrices}). The subgroup of diagonal matrices of order $q$ has size $q^2$, then $n\leq 2$. In $n=2$ then there exists $h\in \mathbb{Z}_q^2$ such that $\rho(h)=kI$ and so $h\in N(\dis(Q)_a)$. Accordingly $Sg(h(a), [a]_{\gamma_Q})=Q\subseteq [a]_{\sigma_Q}$ and then $\sigma_Q=1_Q$, contradiction. Hence $n=1$ and $q>2$ since the two elements quandle is not connected. \\
If $q>2$ divides $p+1$ then by Lemma \ref{eigenvalues of A} and Lemma \ref{faithful action}(i) $\mathbb{Z}_q^n$ is cyclic, i.e. $n=1$.
\end{proof}
\noindent If $Q\in \LSS(p^m,q^n,\sigma_Q)$ and $\sigma_Q\wedge \gamma_Q=0_Q$, then $n$ is even 
according to Theorem \ref{dis mu and S =0}.
\begin{proposition}\label{then q=2}
Let $Q\in \LSS(p,q^{2n},\sigma_Q)$ such that $\sigma_Q\wedge\gamma_Q=0_Q$. Then $\sigma_Q=0_Q$ and $q=2$.  
\end{proposition}

\begin{proof}
The equivalence $\sigma_Q$ is trivial by virtue of Corollary \ref{cor for pq}(ii) and Lemma \ref{sigma_Q=0 iff}.\\ 
Let $q>2$ and $Q\in \LSS(p,q^{2n},0_Q)$. Then $\dis(Q)\cong \mathbb{Z}_p^2\rtimes_{\rho} K$ where $K$ is the extraspecial $q$-group of exponent $q$ and $q$ divides $p-1$. %The center of $K$ normalizes the stabilizer of $\dis_{\gamma_Q}$ (see Corollary \ref{faithful action and sigma =0}). 
So according to Lemma \ref{eigenvalues of A} $\rho_x$ is diagonalizable for every $x\in K$. If the eigenvalues of $\rho_z$ are different, then $K$ embeds in the centralizer of $\rho_z$ which is the subgroup of the diagonal matrices. Since $K$ is not abelian then $\rho_z$ is a scalar matrix. 
%Let $e_1\in \left(\dis_{\gamma_Q}\right)_a$, $e_2$ be a basis of $\dis_{\gamma_Q}$. Since any non-central element $x$ of $K$ does not normalize $\left(\dis_{\gamma_Q}\right)_a$ (see \ref{faithful action and sigma =0}), then $\rho_x$ is not diagonal with respect to the basis $e_1$ and $e_2$. Since $[\rho_x,\rho_z]=1$ then $\rho_z$ is scalar.
%
% and $x\in K$ a non central element then 
%\begin{displaymath}
%\rho_z=\begin{bmatrix}
%\lambda_1 & 0\\
%0 & \lambda_2
%\end{bmatrix},\quad \rho_x=\begin{bmatrix}
%a & c\\
%b & d
%\end{bmatrix}
%\end{displaymath}
%where $c\neq 0$, since $x$ does not normalize $\left(\dis_{\gamma_Q}\right)_a$ (see \ref{faithful action and sigma =0}). Since $[\rho_x,\rho_z]=0$ then $\lambda_1=\lambda_2$. 
%
%so with respect to a basis of eigenvectors of $\rho(x)$ we have
%\begin{displaymath}
%\rho(x)=\begin{bmatrix}
%\mu_1 & 0\\
%0 & \mu_2
%\end{bmatrix},\quad \rho(y)=\begin{bmatrix}
%b & d\\
%c & e
%\end{bmatrix}.
%\end{displaymath}
If $x$ is not a central element of $K$ then $\rho_x$ has different eigenvalues $\mu, \lambda$. Let $y\notin C_K(x)$, then %do not commute the condition $[\rho_y,\rho_x]=\rho_z^l$ implies that 
\begin{eqnarray*}
\rho_y^{-1} \rho_x^{-1}\rho_y\rho_x &=&
\begin{bmatrix}
a & c\\
b & d
\end{bmatrix}^{-1}
\begin{bmatrix}
\mu_1 & 0\\
0 & \lambda
\end{bmatrix}^{-1}
\begin{bmatrix}
a & c\\
b & d
\end{bmatrix}
\begin{bmatrix}
\mu & 0\\
0 & \lambda
\end{bmatrix}=
\frac{(\mu \lambda)^{-1}}{\det(\rho_g)}\begin{bmatrix}
\mu(\lambda ad-\mu bc)	& \lambda(\lambda-\mu)cd\\
\mu(\mu-\lambda)ab	 & \lambda(\mu ad-\lambda bc)
\end{bmatrix} =\rho_z^s.
% &=&\begin{bmatrix}
% k & 0\\
% 0 & k
%\end{bmatrix}.
\end{eqnarray*}
Therefore $a=d=0$ and $\mu=-\lambda$. So $[\rho_y,\rho_x]=-I\in \langle\rho_z\rangle$ which has order $2$. Hence $q=2$. \end{proof}
%
%Using Corollary \ref{cor for pq}, Propositions \ref{then q=2} and \ref{prop of LSS} and Corollary \ref{S=0} we have:
%\begin{corollary}\label{corollary of corollary}\comment{maybe remove it and for $4p$ cite all the other results}
%	Let $Q\in  \LSS(p,2^{2n},0_Q)$. Then $$\dis(Q)\cong \mathbb{Z}_p^2\rtimes_{\rho} H,$$
%	where $H$ is an extraspecial $2$-group and $\rho$ is a faithful action and $Q$ is latin.
%\end{corollary}
\noindent Recall that strictly simple quandles have an affine representation over finite fields given by $\aff(\mathbb{F}_q,\lambda)$ where $\lambda\in \mathbb{F}^*$. In particular the orbits of the bijection $x\mapsto \lambda x$ have all length equal to the multiplicative order of $\lambda$.
\begin{proposition}\label{on order of LeftMul of LSS}
Let $Q\in \LSS(p,2^{2n},0_Q)$ and let $k$ be the order of the left multiplications of $Q/\gamma_Q$. Then $|L_a|=k$ for every $a\in Q$ and $p=1 \pmod k$.
\end{proposition}
\begin{proof}
Let $Q=\mathcal{Q}(G,H,f)$ where $G=\dis(Q)$, $f=\widehat{L}_a$ and $H=Fix(f)$. Let $b\notin [a]$. According to Proposition \ref{prop of LSS}, $Sg(a,b)\cong Q/\gamma_Q$ and then $|b^{L_a}|=|[b]^{L_{[a]}}|=k$, since all the cycles of $L_{[a]}$ have equal length. The block $[a]$ is isomorphic to $\aff(\gamma_2(G)/\gamma_2(G)_a,f|_{\gamma_2(G)})\cong \aff(\mathbb{Z}_p,t)$ for some $t\neq 0,1\in\mathbb{Z}_p$. So restriction of $f$ to $\gamma_2(G)$ is diagonalizable with eigenvalues $1$ and $t\neq 1$. We denote by $F$ the matrix with respect to a basis of eigenvectors. If $x\in K=G/\gamma_2(G)$ is a non central element of $K$ then $\rho_x$ is not diagonal with respect to this basis since $x$ does not normalize $\left(\dis_\alpha\right)_a$ (see \ref{faithful action and sigma =0}). The center of $K$ is fixed by $f_{\gamma_2(G)}$ and its induced mapping over $K/Z(K)\cong G/\gamma_1(G)$ is $f_{\gamma_1(G)}$. Then according to Lemma \ref{order of aut and left mult} the order of $f_{\gamma_2(G)}$ and the order of $f_{\gamma_1(G)}$ coincide and so $f^k(x)=x d$ where $d\in\gamma_2(G)$. Hence
 \begin{equation}\label{commutator condition}
 \rho_{f^k(x)}=F^k\rho_{x}F^{-k}=
 \begin{bmatrix} 
1 & 0\\
0 & t^k
\end{bmatrix} 
\begin{bmatrix}
a & c\\
b & d
\end{bmatrix}
\begin{bmatrix} 
1 & 0\\
0 & t^{-k}\end{bmatrix}=\begin{bmatrix}
a	& ct^{-k}\\
bt^k	 &  d
\end{bmatrix}=
 \rho_{x}=\begin{bmatrix} 
a & c\\
b & d\end{bmatrix}\, .
 \end{equation}
Equation \eqref{commutator condition} implies that $t^k=1$.
%\begin{displaymath}
%F^k\rho_x-\rho_x F^k=\begin{bmatrix}
%1 & 0\\
%0 & t^k
%\end{bmatrix} \begin{bmatrix}
%a & c\\
%b & d
%\end{bmatrix}- \begin{bmatrix}
%a & c\\
%b & d
%\end{bmatrix}\begin{bmatrix}
%1 & 0\\
%0 & l^k
%\end{bmatrix}=\begin{bmatrix}
%0 & c(t^k-1)\\
%-b(t^k-1) & 0
%\end{bmatrix}.
%\end{displaymath}
%Since $x$ does not normalize $\left(\dis_\alpha\right)_a$, then $b\neq 0$, and so $t^k=1$. 
Therefore $p-1=0 \pmod k$ and so $|b^{L_a}|$ divides $k$ whenever $b\in [a]_\alpha$. Therefore $|L_a|=l.c.m.\setof{|b^{L_a}|}{b\in Q}=k$ for every $a\in Q$.
\end{proof}

 According to Proposition \ref{ex on 2to the n p} and Proposition \ref{then q=2} $\LSS(p,q,\gamma_Q)$ and $\LSS(p,4,0_Q)$ are the classes of the smallest examples of faithful LSS subdirectly irreducible quandles. 
% \comment{maybe move to the intro and mention that they almost coincide with non simple connected quandles of size $4p$ and $pq$}

Using the results of this section we can prove a non-existing result for non-affine latin quandles of size $8p$ where $p$ is a prime.
%\begin{lemma}\label{maximal cong for latin}
%Let $Q$ be a finite latin quandle and $\alpha,\beta$ be maximal congruences. Then $Q/\alpha\wedge\beta\cong Q/\alpha\times Q/\beta$.
%\end{lemma}
%
%\begin{proof}
%Clearly $Q/\alpha\wedge\beta$ embeds subdirectly into $Q/\alpha\times Q/\beta$. Both factors are strictly simple, therefore the block $[a]_{\alpha}$ maps onto $Q/\beta$ and then $|Q|=|Q/\alpha||[a]_{\alpha}|\geq |Q/\alpha||Q/\beta|$. Hence $Q\cong Q/\alpha\times Q/\beta$.
%\end{proof}
%
%

\begin{theorem}
Let $p$ be a prime. Latin quandles of size $8p$ are affine.
%There are no non-affine quandles of size $8p$ where $p$ is a prime.
%Let $Q$ be a latin non-affine quandle of size $8p$, $p>5$. Then $Q\in \LSS(8,p,\gamma_Q)$ and $p=7,31$.
\end{theorem}

\begin{proof}
A computer search on the RIG library of GAP reveals that the claim is true for $p\leq 5$.\\ 
Let $p>5$. The quandle $Q$ has no factor of size $2$, and it has also no factor of size $4$, since in such case the blocks of the congruence have size $2p$, but there is no connected quandle of size $2p$ for $p>5$ \cite{McC}. Simple factors of $Q$ have prime power size, so they have either size $8$ or $p$. Let $\alpha$ and $\beta$ be maximal congruences of $Q$ and $\gamma=\alpha\wedge\beta$. Clearly $Q/\gamma$ embeds subdirectly into $Q/\alpha\times Q/\beta$. Both factors are strictly simple, therefore the block $[a]_{\alpha}$ maps onto $Q/\beta$ and then $|Q/\gamma|=|Q/\alpha||[a]_{\alpha}|\geq |Q/\alpha||Q/\beta|$. Hence $Q/\gamma\cong Q/\alpha\times Q/\beta$. So, if $Q$ has two different maximal congruence then $Q$ has a factor of size either $64$, $p^2$ or $8p$. % according to Lemma \ref{maximal cong for latin}. 
Since the size of the factors divides the size of $Q$, the first two cases can not happen and in the last case $Q$ is a direct product of simple latin quandle and then affine. Let $\alpha$ be the unique maximal congruence of $Q$. If the blocks of $\alpha$ have size $p$ then it is also minimal, if the blocks have size $8$, since all the latin quandles of size $8$ are strictly simple then $\alpha$ is minimal aswell. Hence if $Q$ is not affine, it has a unique proper congruence, which is $\gamma_Q$ since $Q$ is not affine and the factor is affine, and then $Q\in \LSS(p,8,\sigma_Q)$ or $Q\in \LSS(8,p,\sigma_Q)$. According to Propositions \ref{ex on 2to the n p} and \ref{then q=2} if $Q\in \LSS(p,2^n,\sigma_Q)$ then $n$ is even. Hence $Q\in \LSS(8,p,\sigma_Q)$ and therefore $\sigma_Q\leq \gamma_Q$ by Corollary \ref{cor for pq}(i). \\
In particular, by Theorem \ref{mu p neq q} we have that $\gamma_2(\dis(Q))=\dis_\alpha=\dis^\alpha=\gamma_1(\dis(Q)$ and it embeds into $\dis([a])^2\cong \mathbb{Z}_2^6$. Therefore $\dis(Q)\cong \dis^\alpha\rtimes \mathbb{Z}_p$ and in particular $p$ divides the size of $\aut{\dis^\alpha}=GL_{3+k}(2)$ for some $0\leq k\leq 3$. Hence $p\in \{7,31\}$. An exhaustive GAP computer search on the groups of the form $\mathbb{Z}_2^{3+k}\rtimes \mathbb{Z}_p$ and their automorphisms for $0\leq k\leq 3$ and $p\in \{7,31\}$ shows that there are no such quandles.
\end{proof}	

\section{Connected quandles of size $pq$}\label{Sec:pq}

%\subsection*{Subdirectly reducible case}
A complete description of connected quandles of size $p^2$ is provided in \cite{Grana_p2}. In this section we give a complete classification of non-simple connected quandles of size $pq$ where $p$ and $q$ are different primes. As we shall see, they are LSS and so we can apply the results of the previous Section. 

First of all note that all connected quandles of size $pq$ are faithful. \begin{lemma}
Connected quandles of size $pq$ are faithful.
\end{lemma}
\begin{proof}
According to \cite[Theorem 1.1]{MeAndPetr} if $Q$ is connected, then $Q/\lambda_Q$ has not prime size. If $|Q|=pq$ every factor of $Q$ has prime size, so $\lambda_Q$ is trivial.\end{proof}
Using that every factor and the blocks of every congruence has prime size blocks it is easy to show that the congruence lattice of connected quandles of size $pq$ is one of the lattices in Figure \ref{lattices2}. Let us first characterize subdirectly reducible quandles of size $pq$.
%
%\begin{lemma}\label{cyclic then central}\comment{remove and cite}
%Let $Q$ be a faithful connected quandle and $\alpha\in Con(Q)$. If $\dis_\alpha$ is cyclic, then $\alpha$ is central.
%\end{lemma}
%\begin{proof}
%Let $\psi:\lmlt(Q)\longrightarrow \aut{\dis_\alpha}$, be the automorphism that define the conjugation action of $\lmlt(Q)$ on $\dis_\alpha$. Since $\aut{\dis_\alpha}$ is abelian, then $\dis(Q)= \gamma_1(\lmlt(Q))\leq ker(\psi)$. So $\dis_\alpha$ is central in $\dis(Q)$ and $\alpha$ is central.
%\end{proof}

\begin{proposition}\label{pq affine} 
%\comment{Maybe: show that $pq$ reducible are LSS. So use general result. Then prove equivalence (i)- (ii)}
Let $p\neq q$ and let $Q$ be a connected quandle $Q$ of size $pq$. The following are equivalent:
\begin{itemize}
\item[(i)] $Q\cong \aff(\mathbb{Z}_p,f)\times\aff(\mathbb{Z}_q,g)$.
%\item[(ii)] $Q$ is principal.
\item[(ii)] $Q$ is subdirectly reducible. 
\end{itemize}
In particular, there are $(p-2)(q-2)$ isomorphism classes of such quandles.
\end{proposition}
\begin{proof}
 (i) $\Rightarrow$ (ii) Clear.

%(ii) $\Rightarrow$ (i) Let $Q=\mathcal{Q}(G,f)$ be a connected principal quandle over $G=\mathbb{Z}_p\rtimes \mathbb{Z}_q$. Then $f_{\gamma_1(G)}=1$ and so $\aff(G/\gamma_1(G),f_{\gamma_1(G)})$ is not connected, contradiction.
%
% In particular $Q$ is latin and so it can not be simple (see \cite{Principal}). Let $\alpha\in Con(Q)$, then $\dis_\alpha$, have prime size, so $\alpha$ is central by Lemma \cite{Principal} and therefore $Q$ is nilpotent ($Q/\alpha$ has prime size and so it is abelian). Therefore  $G$ is a direct product and its $p$ and $q$ components provide trivially intersecting congruences of $Q$.

(ii) $\Rightarrow$ (i) Note that any congruence of $Q$ has block of prime size, so any pair of congruences meet at $0_Q$ and every congruence has strictly simple factor of prime size. Therefore $Q$ is a direct product (see Theorem \ref{LSS reducible}).
%
%
%A presentation of the only non-abelian group of size $pq$ is given by $G=\mathbb{Z}_p\rtimes \mathbb{Z}_q$ is $\langle a, b \ | \ a^p=b^q=1, bab^{-1}=a^m\rangle$, where $2\leq m \leq p-1$ has multiplicative order $q$, and automorphism $f$ maps:
%\begin{displaymath}
%a\mapsto a^k, \quad b \mapsto a^s b^r,
%\end{displaymath}
%for some $0<k\leq p-1$, $0<r\leq q-1$. Thus,
%\begin{displaymath}
%f(bab^{-1})=f(a^{m})=a^{km}=f(b)f(a)f(b)^{-1}= a^{m^r k}.
%\end{displaymath}
%Therefore $m^{r-1}=1$ and so $r=1 \ mod \, q$. The factor of $Q$ given by $F =\mathcal{Q}( G/\gamma_1(G),\widetilde{f})$, where $\widetilde{f}(b\gamma_1(G))=f(b)\gamma_1(G)=b\gamma_1(G)$, is a projection quandle. So $Q$ is not connected.
%
% 
%(i) $\Rightarrow$ (iii) The group $G=\mathbb{Z}_p\times \mathbb{Z}_q$ is the only Abelian group of size $pq$ and it has two characteristic subgroups with trivial intersection. Hence any connected affine quandle $Q=\mathbb{Q}(G,f)$ has two congruences that meet at $0_Q$, since they correspond to $f$-invariant subgroups.
%
%
%(iii) $\Rightarrow$ (i)  If $Con(Q)$ is as in (C) of figure \ref{lattices}, then $Q\cong\mathcal{Q}(\mathbb{Z}_p,f)\times\mathcal{Q}(\mathbb{Z}_q,g)$, so $Q$ is affine. 
%
\end{proof}
Note that quandles in Proposition \ref{pq affine} are connected subdirectly reducible LSS quandles as described in Theorem \ref{LSS reducible}.

Let us consider the subdirectly irreducible case. The unique proper congruence is $\gamma_Q$ since the factor has prime size and then it is abelian, and the blocks have prime size. According to \cite{EGS} they are either connected or projection. In the first case $Q$ is a LSS quandle and $Q\in \LSS(p,q,\gamma_Q)$ (see Corollary \ref{cor for pq}). In the second case we have the following:

\begin{proposition}\label{proj blocks}
Let $Q$ be a  subdirectly irreducible connected quandle of size $pq$. If the blocks of $\gamma_Q$ are projection subquandles then $Q\in \LSS(2,3,\gamma_Q)$.
\end{proposition}
\begin{proof}
	%\comment{just cite}
%Let $G=\dis(Q)$. The subgroup $K_{[a]}=\langle L_c L_a^{-1}, \, c,\in [a]\rangle$ is a non-trivial normal subgroup of $\dis_{\gamma_Q}$ since $Q$ is faithful and it fixes $[a]$. The quandle $Q$ is generated by $[a],b$ for every $b\notin [a]$, then $K_{[a]}\bigcap K_{[b]}\leq K_{[a]}\bigcap G_b=1$ and the it acts regularly on the block $[b]$. Therefore, $ K_{[a]}\cong\mathbb{Z}_p$. All the subgroups $K_{[a]}$  are normal in $\dis_{\gamma_Q}$ and they commute, then they generate $\dis_{\gamma_Q}$ which is therefore abelian and then $\gamma_Q$-semiregular. 
The quandle $Q$ is a faithful connected extension of a strictly simple quandle by a projection quandle of prime power size. Then \cite[Theorem 4.3]{Principal} applies and since $\gamma_1(G)=\gamma_2(G)=\dis_{\gamma_Q}$ (by Proposition \ref{sims}) then $\gamma_Q=\sigma_Q$ and so $|Q|=6$. As we already pointed out before $Q\in \LSS(2,3,\gamma_Q)$.
\end{proof}

From now on we assume that $p\geq 3$ (since two elements quandles are projection). The following Theorem characterizes quandles in $\LSS(p,q,\gamma_Q)$ by automorphisms of the group $G_{p,q}\cong \mathbb{Z}_p^2\rtimes_\rho \mathbb{Z}_q$ with $\det{\rho(1)}=1$ which are  investigated in the Appendix \ref{Sec:appendix2}. In particular recall that $\LSS(p,q,1_Q)$ is empty (see Corollary \ref{cor for pq}(ii)). 
%We are using some group theoretical results given in the Appendix \ref{Sec:appendix}. 
\begin{theorem}\label{automorf for quandles}
Let  $Q$ be a quandle. The following are equivalent: 
\begin{itemize}
\item[(i)] $Q\in \LSS(p,q,\gamma_Q)$.
\item[(ii)] $Q\cong \mathcal{Q}(G_{p,q},Fix(f),f)$, where $f_{\gamma_1(G_{p,q})}=-1$, $det(f|_{\gamma_1(G_{p,q})})=-1$, $Tr(f|_{\gamma_1(G_{p,q})})=0$ and $\dis(Q)\cong G_{p,q}$.
\end{itemize}
In particular $Q$ is latin.
%In particular the blocks of $[1_Q,1_Q]$ and $Q/[1_Q,1_Q]$ are involutory.
\end{theorem}
\begin{proof}
(i) $\Rightarrow$ (ii) %The quandle $Q$ is not principal by Proposition \ref{pq affine}, so 
According to Theorem \ref{mu p neq q} and Corollary \ref{cor for pq}(i) we have that $\dis(Q)\cong G= \mathbb{Z}_p^2\rtimes_\rho \mathbb{Z}_q$, where $\rho$ is a faithful action with $\rho(1)=A$ and $Z(G)=1$ and $q$ divides $p^2-1$. So $Q\cong \mathcal{Q}(G,Fix(f),f)$ for some $f\in \aut{G}$ which is identified by
\begin{equation}\label{not aut}
f = \begin{cases}
f|_{\gamma_1(G)} =F,\quad 
c\mapsto  c^{u_f} d_f        \,,
\end{cases}
\end{equation}
where $F\in GL_2(p)$, $1\leq u_f\leq q-1$ and $d_f\in \gamma_1(G)$. Note that $FAF^{-1}=A^{u_f}$, since $f$ is an automorphism and in particular $det(A)^{u_f}=det(A)$ and $det(A)^q=1$. If $det(A)\neq 1$ then $u_f=1 \pmod q$ and accordingly $Q/\gamma_Q\cong \aff(G/\gamma_1(G),f_{\gamma_1(G)})$ is projection, contradiction. Therefore $det(A)=1$ and then $G\cong G_{p,q}$ by virtue of Proposition \ref{uniqueness of the group}. According to Proposition \ref{aut of G_A}, $u_f=-1$.\\
The subgroup $H=Fix(f)$ is not normal (see Remark \ref{existence of aut with desired prop}) and then $Core_{G_{p,q}}(H)=1$, so in particular $Fix(f)$ is not invariant under $A$. %According to \cite[Section 2]{GB}, 
Accordingly $\dis(Q)\cong G_{p,q}/Core_{G_{p,q}}(H)=G_{p,q}$.\\ 
The restriction $F$ is diagonalizable with eigenvalues $1, k\neq 0,1$. Since $f\in \aut{G}$ then $FA=A^{-1}F$ (*). Using a basis of eigenvectors of $F$, $a\in H$ and $b$ condition (*) implies that $k=-1$ and so $Tr(F)=0$ and $det(F)=-1$.
%\begin{eqnarray*}
%FA(a)&=&f(a^\alpha b^\beta)=A^{-1}F(a)=A^{-1}(a)\\
%FA(b)&=&f(a^\gamma b^\delta)=a^{\gamma} b^{k \delta}=A^{-1}f(b)=A^{-1}(b)^{k}.
%\end{eqnarray*}
%Then, $A^{-1}(a)=a^{\alpha}b^{\beta k}$ and $A^{-1}(b)=a^{\frac{\gamma}{k}}b^{\delta}$. Since $\beta\neq 0$ then $k=-1$ and so $Tr(F)=0$ and $det(F)=-1$. 

(ii) $\Rightarrow$ (i) Let $Q\cong \mathcal{Q}(G,H,f)$ where $G=G_{p,q}$, $H=Fix(f)$, $u_f=-1$ and $Tr(F)=0$ and $det(F)=-1$. The matrix $F$ has eigenvalues $1$ and $-1$, so $H\cong \mathbb{Z}_p$ and $|Q|=pq$. Moreover $Q$ has a congruence with strictly simple connected factor and strictly simple connected blocks and so it is connected and LSS (see Proposition \ref{LSS iff}).  

In order to show that $Q$ is latin, note that the mapping \eqref{right mult} in Remark \ref{existence of aut with desired prop}(ii) is the right multiplication by $H\in Q$ and then $Q$ is latin. 
 \end{proof}

We show that there exists just two subdirectly irreducible connected quandles of size $pq$ up to isomorphism using the isomorphism Theorem \ref{iso theorem1}. By virtue of Theorem \ref{automorf for quandles} the quandle $\mathfrak{Q}_d=\mathcal{Q}(G_{p,q},Fix(f_d),f_d)$, where
\begin{equation}\label{non-involutory}
f_d = 
    \begin{cases}
    F=\begin{bmatrix}
    0 & 1\\
    1 & 0\\ 
    \end{bmatrix}, \quad c\mapsto  c^{-1} d,       
    \end{cases}.
\end{equation}
and $d\in \gamma_1(G_{p,q})$ are latin quandles of size $pq$.

\begin{theorem}\label{iso_class_invol}
Let $p,q\geq 3$ be primes such that $p^2=1\pmod q$ and $1\neq a\in \gamma_1(G_{p,q})$. The quandles $\mathfrak{Q}_1$ and $\mathfrak{Q}_a$ are the unique subdirectly irreducible connected quandle of size $pq$ up to isomorphism. In particular, they are the unique non-affine latin quandles of size $pq$.
\end{theorem}
\begin{proof}
Let $q$ divides $p-1$. According to Lemma \ref{eigenvalues of A}, $A$ is diagonazable with eigenvalues $k,k^{-1}$ (Lemma \ref{eigenvalues of A}) and so the matrices $F$ such that $Tr(F)=0$, $det(F)=-1$ and $FAF=A^{-1}$, are of the form
	\begin{eqnarray*}
		F=     \begin{bmatrix}
			0 & x^{-1}\\
			x & 0
		\end{bmatrix}, \quad\text{where $x\neq 0$}.  
	\end{eqnarray*}
	So the number of automorphism of $G_{p,q}$ as in Theorem \ref{automorf for quandles} is $p^2(p-1)$. It is easy to compute the centralizer of $f_1$ as
	\begin{equation}
	h = 	\begin{cases}
	H=\begin{bmatrix}
	x & 0\\
	0 & x^{-1}\\ 
	\end{bmatrix}, \quad c\mapsto  c d,     \quad d\in \ker{1-AF}  
	\end{cases},\quad 
	h = 	\begin{cases}
	H=\begin{bmatrix}
	0 & x^{-1}\\
	x & 0\\
	\end{bmatrix}, \quad c\mapsto  c^{-1} d,     \quad d\in \ker{1-AF}  
	\end{cases}
	\end{equation}
and then it has size $2(p-1)p$. Let $a=(1,0)$. Likewise the centralizer of $f_a$ is
	\begin{equation}
	h = 
	\begin{cases}
	H=F, \quad c\mapsto  c d,     \quad d=(u,-ku), \,u\in\mathbb{Z}_p  
	\end{cases}\quad 
	h = 
	\begin{cases}
	H=F, \quad c\mapsto  c^{-1} d,    \quad d=(1-ku,u),\, u\in\mathbb{Z}_p
	\end{cases}
	\end{equation}
	and then it has size $2p$. Then $f_1$ and $f_a$ are not conjugate since their centralizer have different orders. Then
	\begin{eqnarray*}
		[\aut{G_{p,q}}:C_{\aut{G_{p,q}}}(f_1)]+[\aut{G_{p,q}}:C_{\aut{G_{p,q}}}(f_a)]&=&\frac{2p^2(p-1)^2}{2p(p-1)}+\frac{2p^2(p-1)^2}{2p}=p^2(p-1)%\\
%		&=& p(p-1)+p(p-1)^2=p^2(p-1)
	\end{eqnarray*}
so $f_1$ and $f_a$ are representatives of conjugacy classes of the desired automorphisms.\\
If $q$ divides $p+1$ the same enumeration argument will follow taking into account that in such case the matrix $A$ in a suitable basis is 
	$$A=\begin{bmatrix}
	0 & -1\\
	1 & -b\\
\end{bmatrix}$$
and $x^2 +bx+1$ is an irreducible polynomial since $A$ has no eigenspaces (see \ref{eigenvalues of A}).

Simple latin quandle have size a power of a prime \cite{Principal,Stanos}, therefore $\mathfrak{Q}_1$ and $\mathfrak{Q}_a$ are the unique latin quandles of size $pq$. 
\end{proof}
The only involutory strictly simple quandles are $\aff(\mathbb{Z}_p,-1)$. So if $Q\in \LSS(p^m,q^n,\sigma_Q)$ is involutory, then $n=m=1$ and $Q\cong \mathfrak{Q}_1$ as defined in \eqref{non-involutory}.\\
According to the enumeration above, the (non-simple) connected subdirectly irreducible quandle of size $3p$ with $p>2$ are the Galkin quandles defined in \cite{C3,C4}.

As a byproduct we obtain the classification of non-associative Bruck loops of order $pq$ already shown in \cite{BruckPetr}, by virtue of the one to one correspondence between involutory latin quandles and Bruck loops of odd order \cite{Stanos}. Indeed $\mathfrak{Q}_1$ is the unique non-affine involutory latin quandle of size $pq$.

\begin{corollary}
Let $p,q\geq 3$ be primes such that $p^2=1\pmod q$. There exists a unique non-associative Bruck loop of size $pq$ up to isomorphism.
\end{corollary}

\section{Connected quandles of size $4p$}\label{sec 4p}
%\subsection*{Subdirectly reducible case}

%
In this section we classify all non-simple connected quandles of size $4p$ and we show that all but a few small exceptions are examples of LSS connected quandles. \\
If $p>2$ then connected quandles of size $4p$ are faithful.
\begin{lemma}
Connected quandles of size $4p$ with $p>2$ are faithful. 
\end{lemma}
\begin{proof}
Assume that $Q$ is not faithful. Then $Q/\lambda_Q$ has not prime size (see \cite[Theorem 1.1]{MeAndPetr}) so $Q/\lambda_Q$ is the unique connected quandle of size $4$. According to \cite[Example 8.3]{MeAndPetr} the orbits with respect to $\lmlt(Q)$ of every element of $Q$ has size $8$, therefore $Q$ has size $8$.
\end{proof}

If $p>5$ there are no connected quandles of size $2p$ \cite{McC,HSV}, so every factor of a connected quandle of size $4p$ has size either $4$ or $p$ and consequently also the blocks of congruences have size either $4$ or $p$. Therefore the congruence lattice is one of the lattices in Figure \ref{lattices2}.  In particular if $Q$ is subdirectly reducible then it is a direct product (the same argument of \ref{pq affine} applies).\\
Recall that there exists a unique connected quandle of size $4$, i.e. $\aff(\mathbb{Z}^2_2,f)$ where the order of $f$ is $3$ and it is strictly simple. 

\begin{proposition}
Let $Q$ be a connected quandle of size $4p$ with $p>5$. The following are equivalent:
\begin{itemize}
\item[(i)] $Q$ is subdirectly reducible.
\item[(ii)] $Q\cong \aff(\mathbb{Z}_2^2,f)\times \aff(\mathbb{Z}_p,g)$.
%\item[(iii)] $Q$ is principal.
In particular there are $p-2$ such quandles.
\end{itemize}
\end{proposition}
%\begin{proof}
%The implications (i) $\Rightarrow$ (ii) and (ii) $\Rightarrow$ (iii) are clear (the same argument of \ref{pq affine} applies). 
%
%(iii) $\Rightarrow$ (i) Let $Q$ be principal and let $H=\dis(Q)/\gamma_1(\dis(Q))$. If $H\cong  \mathbb{Z}_p$, then $\dis(Q)\cong \mathbb{Z}_2^2\rtimes \mathbb{Z}_p$, and then $p$ divides $6$ or $Q$ is a direct product. If $H\cong \mathbb{Z}_2^2$ then $\dis(Q)\cong \mathbb{Z}_p\rtimes \mathbb{Z}_2^2$ and it has non-trivial center. Hence $\zeta_Q \neq 0_Q$ and the $Q$ is a direct product of a quandle of size $p$ and a quandle of size $4$ \cite[Theorem 1.4]{CP}. 
%\end{proof}
%
%\subsection{Subdirectly irreducible case}

We turn now out attention to the subdirectly irreducible case. 
\begin{lemma}
Let $Q$ be a subdirectly irreducible connected quandle of size $4p$ with $p>5$. Then $\gamma_Q$ is the unique congruence of $A$, $\zeta_Q=0_Q$ and $Z(\dis(Q))=1$.
\end{lemma}

\begin{proof}
	If $\zeta_Q\neq 0_Q$ then $Q$ is faithful and nilpotent and then it decomposes as a direct product of prime power size quandles \cite[Theorem 1.4]{CP}. Hence $Q$ is subdirectly reducible, contradiction. The orbits of $Z(\dis(Q))$ are contained in the blocks of $\zeta_Q$. Accordingly $Z(\dis(Q))=1$.
In particular $Q$ is not abelian and then its unique proper congruence is $\gamma_Q$ since the factor is abelian. 
\end{proof}

Assume that the unique factor of $Q$ has prime size and the blocks have size $4$. So that the blocks of $\gamma_Q$ are isomorphic to one of the following quandles: $\mathcal{P}_4$, $\aff(\mathbb{Z}_4,-1)$ or $\aff(\mathbb{Z}^2_2,f)$ (the unique connected quandle of size $4$), since these are all the homogenous quandles of size $4$. 
%
%\begin{lemma}
%Let $Q$ be a S.I. latin quandle of size $4p$ with $p>5$. Then:
%\begin{itemize}
%\item[(i)] $Q$ is solvable;
%\item[(ii)] $Q=Sg(a,b,c)$, wherever $[a]_\mu=[b]_\mu\neq [c]_\mu$.
%%\item[(iii)] $\dis(Q)_{a,b}=\{1\}$ whenever $[a]_\mu\neq [b]_\mu$.
%\end{itemize}
%\end{lemma}
%
%\begin{proof}
%(i) Since $\dis(Q)^\mu\hookrightarrow \aut{[a]_\mu}$ which is solvable. 
%
%(ii) It follows since the factor and the blocks are two-generated subquandles.
%\end{proof}

\begin{proposition}\label{dis_4p}
Let $Q$ be a subdirectly irreducible connected quandle of size $4p$ and $|Q/\gamma_Q|=p>5$. Then $p= 7$.
%
%Let $Q$ be a S.I. connected quandle of size $4p$ with $p>5$ and let $|Q/[1,1]|=p$. Then:
%\begin{itemize}
%\item[(i)] $\dis^{[1,1]}=\dis_{[1,1]}=\gamma_1(\dis(Q))=\gamma_2(\dis(Q))$;
%\item[(ii)] $Z(\dis(Q))=\{1\}$.
%\item[(iii)] $\dis(Q)\cong A\rtimes \mathbb{Z}_p$, where $A=\mathbb{Z}_2^{2+k}$ and $0\leq k\leq 2$ or $A=\mathbb{Z}_4^2$. 
%\end{itemize}
\end{proposition}
\begin{proof}
Let $G=\dis(Q)$. According to Proposition \ref{sims}, $\gamma_1(G)=\gamma_2(G)=\dis_{\gamma_Q}$ since $\dis(Q)$ is not nilpotent and $\dis_{\gamma_Q}$ is minimal in $Norm(Q)$. Let us discuss all the possible cases according to the properties of the blocks of $\gamma_Q$: 
\begin{itemize}
\item[(i)] If they are are connected, then $Q$ is LSS and we can apply Lemma \ref{embedding of ker} and $\gamma_1(G)$ embeds into $\mathbb{Z}_2^4$.

\item[(ii)] %\comment{check. Maybe find an argument to show that this cannot happen. It never happens into the examples}. 
If the blocks of $\gamma_Q$ are isomorphic to $\aff(\mathbb{Z}_4,-1)$. then $\gamma_1(G)|_{[a]}$ embeds into $\aut{[a]}\cong \mathbb{Z}_4\rtimes \mathbb{Z}_2$, see \cite[Theorem 2.1]{Elham}. Therefore, $\gamma_1(G)=\dis_{\gamma_Q}$ is solvable and therefore it is abelian (see \ref{minimal congruences}). Hence $\gamma_1(G)$ is $\gamma_Q$-semiregular, so $|\gamma_1(G)|_{[a]}|= 4$ and then $\gamma_1(G)$ embeds into $\mathbb{Z}_4^2$ or into $\mathbb{Z}_2^4$. %\comment{to check: it never happens into the examples.}
%is ${[1,1]}$-semiregular and so $h|_{[a]}(x)=\lambda x +c_h$. So $h$ has no fixed points if and only if $\lambda=3$ and $h,g\in \dis_{[1,1]}$ commutes if and only if $c_g-c_h=2$ and thus $|h_{[a]}|=2$. Then $\dis_{[1,1]}|_{[a]}\cong \mathbb{Z}_2^2$ and it embeds into $\mathbb{Z}_2^4$. 
%(iii) Since $\dis_\mu=\dis^\mu$, then $\dis_\mu$ embeds into $\mathbb{Z}_2^k$ and $\dis(Q)_a$ is regular on each block different from $[a]$, so $|\dis(Q)_a|=2,4$. Accordingly, $\dis_\mu \cong \mathbb{Z}_2^{2}$, where $k=3,4$. 

\item[(iii)] If the blocks of $\gamma_Q$ are projection subquandles, then by Lemma \cite[Lemma 4.1]{Principal} $\gamma_Q$ is abelian and $\gamma_1(G)$ is transitive and semiregular on each block of $\gamma_Q$. Therefore $|\gamma_1(G)|_{[a]}|=4$, so $\gamma_1(G)$ embeds into $\mathbb{Z}_4^{2}$ or into $\mathbb{Z}_2^4$.
\end{itemize}
Hence, $\dis(Q)\cong A\rtimes \mathbb{Z}_p$, and $p$ divides the size of $\aut{A}$ where $A=\mathbb{Z}_2^{2+k}$ and $k\leq 2$ or $A=\mathbb{Z}_4^k$ and $k\leq 2$. In the first case $p\in \{2,3,5,7\}$ and in the second case $p\in \{2,3\}$. \end{proof}

A complete list of subdirectly irreducible connected quandles of size $28$ which falls in the class described in Proposition \ref{dis_4p} is given in the following table (none of them is latin):

\medskip
\begin{center}

\begin{small}{\renewcommand\arraystretch{1.2} 
\begin{tabular}{ | c | c| c | c| c| c|}
\hline
RIG & $|Q/\gamma_Q|$ & $[a] $ &$\dis(Q)$&  Note \\
\hline
%(8,2) & 4 & $\mathcal{P}_2 $ & $\mathcal{Q}_8$ & $\LSS(2,4,1_Q)$ \\
%(12,1) & 6 & $\mathcal{P}_2$ & $A_4$ & non faithful\\
%(12,2) & 6 & $\mathcal{P}_2$ & $A_4$ & non faithful\\
%(12,5) & 6 & $\mathcal{P}_2$ & $\mathbb{Z}_4^2\rtimes \mathbb{Z}_3$ & faithful \\
%(12,6) & 6 & $\mathcal{P}_2$ & $\mathbb{Z}_4^2\rtimes \mathbb{Z}_3$ &faithful\\
%(12,7) & 6 & $\mathcal{P}_2$ & $\mathbb{Z}_4^2\rtimes \mathbb{Z}_3$ &faithful\\
%(12,10) & 4 & $\mathcal{P}_3$ & $\mathbb{Z}_3^2\rtimes \mathcal{Q}_8$ &faithful\\
%(20,3) & 10 & $\mathcal{P}_2$ & $A_5$ & non faithful\\
%(20,5) & 5 & $\mathcal{P}_4$  & $\mathbb{Z}_2^4\rtimes \mathbb{Z}_5$ &faithful\\
%(20,6) & 5 & $\mathcal{P}_4$  & $\mathbb{Z}_2^4\rtimes \mathbb{Z}_5$ &faithful\\
%(20,9) & 10 & $\mathcal{P}_2$  & $\mathbb{Z}_2^4\rtimes A_5$ &faithful\\
%(20,10) & 5 & $\mathcal{P}_4$  & $\mathbb{Z}_2^4\rtimes A_5$ &faithful\\
(28,3) & 7 & $\aff(\mathbb{Z}_2^2,f)$ & $\mathbb{Z}_2^3\rtimes \mathbb{Z}_7$ & $\LSS(4,7,\gamma_Q$)\\
(28,4) & 7 & $\aff(\mathbb{Z}_2^2,f)$ & $\mathbb{Z}_2^3\rtimes \mathbb{Z}_7$ & $\LSS(4,7,\gamma_Q$)\\
(28,5) & 7 & $\aff(\mathbb{Z}_2^2,f)$ & $\mathbb{Z}_2^3\rtimes \mathbb{Z}_7$ &$\LSS(4,7,\gamma_Q$)\\
(28,6) & 7 & $\aff(\mathbb{Z}_2^2,f)$ & $\mathbb{Z}_2^3\rtimes \mathbb{Z}_7$ &$\LSS(4,7,\gamma_Q)$\\
\hline
\end{tabular}}

\end{small}
\end{center}
\medskip

\noindent Let us consider the case in which $Q/\gamma_Q$ is the unique connected quandle of size $4$, i.e. $\aff(\mathbb{Z}_2^{2},f)$.

\begin{proposition}
Let $Q$ be a connected subdirectly irreducible quandle of size $4p$, $p> 5$ and $|Q/\gamma_Q|=4$. Then $Q\in\LSS(p,4,0_Q)$.\end{proposition}
\begin{proof}
The blocks of $\gamma_Q$ are homogeneous quandles of prime size. So they are either projection or connected. Assume that the blocks of $\gamma_Q$ are projection. Then the quandle $Q$ is a connected faithful extension of a strictly simple factor by a prime size projection quandle. So the size of $Q$ is either $6$ or $12$ according to \cite[Theorem 4.8]{Principal}. Therefore the blocks are connected strictly simple quandle and then $Q\in \LSS(p,4,\sigma_Q)$. According to Proposition \ref{ex on 2to the n p} and to Proposition \ref{then q=2} $\sigma_Q=0_Q$. 
\end{proof}

\begin{proposition}\label{dis of 4p}
Let $Q\in \LSS(p,4,0_Q)$. Then $Q$ is latin and $\dis(Q)\cong \mathbb{Z}_p^2\rtimes_\rho \mathcal{Q}_8$ where $\rho$ is a faithful action. Moreover $L_a^3=1$ for every $a\in Q$ and $p=1 \pmod 3$.
\end{proposition}
\begin{proof}
According to Theorem \ref{dis mu and S =0} and Corollary \ref{cor for pq}(ii) $\dis(Q)\cong G=\mathbb{Z}_p^2\rtimes_\rho K$ where $K$ is an extraspecial $2$-group of size $8$ and $\rho$ is a faithful action. 
Let $f=\widehat{L_a}$. According to Lemma \ref{order of aut and left mult} and Lemma \ref{on order of LeftMul of LSS} $|f|=|L_a|=|L_{[a]}|=3$ for every $a\in Q$ and $p=1\pmod 3$. So, the automorphism $f_{\gamma_2(G)}\in \aut{K}$ has also order $3$. The automorphisms group of $\mathcal{D}_8$ has size $8$, then it has no element of order $3$, so $K\cong \mathcal{Q}_8$. Moreover $Q$ is latin by Proposition \ref{prop of LSS}.
%
% The order of automorphism of $\mathcal{Q}_8$ are $2,3$ and $4$, so the order of the induced on $H$ is $3$. Then $\rho_{f^3(x)}=f^3\rho_{x}f^{-3}=\rho_{x}$ for every $x\in H$. Given a basis of eigenvectors of $f$, we have:
%\begin{displaymath}
%f^3\rho_x-\rho_x f^3=\begin{bmatrix}
%1 & 0\\
%0 & k^3
%\end{bmatrix} \begin{bmatrix}
%a & c\\
%b & d
%\end{bmatrix}- \begin{bmatrix}
%a & c\\
%b & d
%\end{bmatrix}\begin{bmatrix}
%1 & 0\\
%0 & k^3
%\end{bmatrix}=\begin{bmatrix}
%0 & c(k^3-1)\\
%-b(k^3-1) & 0
%\end{bmatrix}.
%\end{displaymath}
%Since $N(\dis(Q)_a)=\dis(Q)_a$, then $b\neq 0$, and so $k^3=1$. Then $f^3=1$ and $p-1=0 \pmod 3$.
\end{proof}
%
%\begin{corollary}
%Let $Q=Q(p,2,1,2,S_Q)$. Then $p-1=0$ mod 3.
%\end{corollary} 

%\begin{lemma}\comment{follows as a corollary}
%Let $Q$ be a connected S.I. quandle of size $4p$ with $|Q/\mu|=4$. Then $S=0_Q$, i.e. $N(\dis(Q)_a)=\dis(Q)_a$.
%\end{lemma}
%\begin{proof}
%Since $|\dis^\mu|=2p^2$, then $|\dis(Q)_a|=2p$. Assume that $S\neq 0_Q$ and then $S\wedge \mu=0_Q$ i.e. for every $[b]_\alpha$ there exists $a_b \, S\, a$. Then $\dis(Q)_a$ embeds into $\mathbb{Z}_{p-1}^4$, so $p$ divides $p-1$, contraddiction.
%\end{proof}
%
%\begin{lemma}\comment{follows as a corollary}
%Let $Q$ be a connected S.I. quandle of size $4p$ with $|Q/\mu|=4$. Then $Sg(a,b)\cong Q/\mu$ for every $[a]_\mu \neq [b]_\mu$ and $Q$ is latin.
%\end{lemma}

%\begin{proof}
% Let assume that $Q=Sg(a,b)$ for some $a,b\in Q$. Then $Q=Sg(c,d)$ for every $c\, \mu \, a$ and $d\, \mu \, b$ since $\dis(Q)_a$ is transitive on $[b]_\mu$ and $\dis(Q)_b$ is transitive on $[a]_\mu$. 
%
%Let $z$ be a generator of $Z(H)$. Then $z|_{[a]_\mu}\in \aut{[a]_\alpha}$. Indeed let $z^2(x)=t+kt+k^2 x=x$ for every $x\in [a]_\mu$. So $t(k+1)=0$, and in both cases $z$ has a fixed point. So if $Q$ is two-generated, $z=1$, contraddiction.
%
%Finally, by \comment{cite LSS} $Q$ is latin.
%\end{proof}
%
%\begin{corollary}\comment{follows as a corollary}
%Let $Q$ be a connected S.I. quandle of size $4p$ with $|Q/\mu|=4$. Then $|a^{L_b}|=3$ whenever $[a]_\mu\neq [b]_\mu$.
%\end{corollary}

Faithful representations of dimension $2$ of $\mathcal{Q}_8$ are irreducible (see Lemma \ref{faithful action}(ii)) and they are all isomorphic \cite{Mayr}. Let $p$ be a prime such that $p=1 \pmod 3$ and let $k\in \mathbb{Z}_p^*$ be an element of order $3$. We define $G_k=\mathbb{Z}_p^2\rtimes_\rho \mathcal{Q}_8$ where
\begin{displaymath}
\rho_x=\begin{bmatrix}
0 & -1\\
1 & 0
\end{bmatrix},\quad
\rho_{y}=\begin{bmatrix}
k^2 & k\\
k & -k^2
\end{bmatrix},\quad \rho_z=\begin{bmatrix}
-1 & 0\\
0 & -1
\end{bmatrix}.
\end{displaymath}
In particular $G_k$ is the unique semidirect product $\mathbb{Z}_p^2\rtimes_\rho \mathcal{Q}_8$ where $\rho$ is a faithful action up to isomorphism.
%\comment{Note that even if $\rho_x$ and $\rho_y$ are diagonalizable, their eigenvectors are different from the eigenvectors of $f$, since they do not commute. Then $Fix(f)$ is not normal in $\dis(Q)$.}

\begin{theorem}\label{4p iff}
Let $p$ be a prime such that $p=1 \pmod 3$ and $G_k=\mathbb{Z}_p^2\rtimes_\rho \mathcal{Q}_8$. The following are equivalent:
\begin{itemize}
	\item[(i)] $Q$ is a connected subdirectly irreducible quandle of size $4p$.
	\item[(ii)] $Q\cong \mathcal{Q}(G_k,Fix(f),f)$ where $f^3=1$, $|Fix(f)|=2p$, $f$ acts irreducibly on $G_k/\gamma_1(G_k)$ and $\dis(Q)\cong G_k$. 
\end{itemize}

\begin{proof}
(i) $\Rightarrow$ (ii) It follows from Proposition \ref{dis of 4p} by virtue of Lemma \ref{order of aut and left mult}.

(ii) $\Rightarrow$ (i) Let $G=G_k$ and $[G,f]$ be the subgroup generated by $\setof{gf(g)^{-1}}{g\in G}$. Recall that the action of $\dis(Q)$ is the canonical left action of $[G,f]$ over $G/Fix(f)$ and $\dis(Q)\cong [G,f]/Core_G(Fix(f))$ (see \cite[Section 2]{GB}). If $f_{\gamma_1(G)}$ acts irreducibly, then 
$$[G/\gamma_1(G),f_{\gamma_1(G)}]=Im(1-f_{\gamma_1(G)})=G/\gamma_1(G).$$ 
Since $Z(K)=\Phi(K)$ and $K/Z(K)\cong G/\gamma_1(G)$ then $K$ is generated by $\setof{gf_{\gamma_2(G)}(g)^{-1}K}{g\in G}$ as well, i.e. $G/\gamma_2(G)=K=[K,f_{\gamma_2(G)}]$. The dimension of $Fix(f)\cap \gamma_2(G)$ is one, then $\gamma_2(G)\cap [G,f]$ is a non-trivial normal subgroup of $G$. The action $\rho$ has no invariant subspaces then $\gamma_2(G)\leq [G,f]$ and $Core_{G}(Fix(f))=1$ (since normal subgroups of $G$ contained in $\gamma_2(G)$ are the invariant subspace under $\rho$). Therefore $G=[G,f]\cong \dis(Q) $ and its action is transitive. Then $Q$ is connected and it has size $[G_k:Fix(f)]=4p$.
\end{proof}
\end{theorem}

According to the isomorphism Theorem \ref{iso theorem1} we need to compute conjugacy classes of automorphisms of $G_k$ satisfying the condition of Theorem \ref{4p iff}. Note that automorphisms of $G_k$ are uniquely determined by their action on generators, and in particular their restriction to the derived subgroup of $G_k$ can be represented by a matrix of rank $2$. A complete description of the automorphisms of such group is given in the Appendix \ref{Sec:appendix}.

\begin{theorem}\label{iso class 4p}
	Let $p>5$ be a prime such that $p=1\pmod 3$. The quandles $\mathfrak{Q}_\lambda=\mathcal{Q}(G_k,Fix(f(\lambda)),f(\lambda))$ where
\begin{equation*}%\label{autom}
		f(\lambda)=\begin{cases}
		x \mapsto y, 	\quad y \mapsto xyz,\quad
		  z\mapsto z,\quad 		
		f(\lambda)|_{\gamma_1(G_k)} =\begin{bmatrix}
		 	1+\lambda & -k(1+\lambda)\\
		 k^2(1+\lambda) & 0	\end{bmatrix}	
				\end{cases}
		\end{equation*}
	and $\lambda=k,k^2$ are the unique quandles in $\LSS(p,4,0_Q)$ up to isomorphism. In particular, %$\mathcal{Q}(G_k,Fix(f(\lambda)),f(\lambda))$
	they are the unique non-affine latin quandles of size $4p$.
%, where
%\begin{displaymath}
%f=\begin{cases}
%x \mapsto y, \quad 		f|_{\gamma_1(G_k)} =\begin{bmatrix}
%1+\lambda & -k^2\lambda\\
%k\lambda & 0	\end{bmatrix}		\\
%y \mapsto xyz\\
%z \mapsto z
%
%\end{cases}.
%\end{displaymath}
\end{theorem}

\begin{proof}
According to Theorem \ref{iso theorem1} we need to compute the conjugacy classes of automorphisms as in Theorem \ref{4p iff}. The number of such automorphisms is $16p$ according to Lemma \ref{aut of Gk_2}. The automorphisms $f(k)$ and $f(k^2)$ are not conjugate, since their restriction to $\gamma_2(G_k)$ have different determinant. Then using Lemma \ref{centralizer} we have
\begin{displaymath}
\frac{|\aut{G_k}|}{|C_{\aut{G_k}}(f(k))|}+
\frac{|\aut{G_k}|}{|C_{\aut{G_k}}(f(k^2))|}=2\frac{24p^2(p-1)}{3p(p-1)}=16p.
\end{displaymath}
Hence $f(k), f(k^2)$ are the representative of the conjugacy classes of the desired automorphisms.

 Since simple latin quandles have prime power size, $\mathfrak{Q}_\lambda$ with $\lambda=k,k^2$ are the unique non-affine latin quandles of size $4p$ (a GAP computer search reveals that for $p\leq 5$ latin quandles of size $4p$ are affine).
\end{proof}

\subsection*{Abelian extensions of the four-elements latin quandle}\label{Sec:ab ext}

Let $Q$ be a connected quandle and $\alpha\in Con(Q)$.	Then $Q$ in an \textit{extension of $Q/\alpha$ by $[a]$} and it is isomorphic to $Q/\alpha\times_\beta [a]=(Q/\alpha\times [a],*)$ where $(a,s)*(b,t)=(a*b,\beta(a,b,s)(t))$ and $\beta:Q\times Q\times [a]\to \sym{([a])}$ is called {\it dynamical cocycle} (the permutations $\beta(a,b,s)$ satisfy suitable conditions, see \cite[Section 2]{AG}). Moreover for every $\gamma:Q\to \sym{(S)}$, $$\sigma(a,b,s)=\gamma(a*b)\beta(a,b,\gamma(a)^{-1}(s))\gamma(b)^{-1}$$ is a dynamical cocycle and $Q\times_\beta S\cong Q\times_\sigma S$.

We turn our attention to a particular family of cocycles. An {\it abelian cocycle} of $Q$ with coefficients in an abelian group $A$ is a triple $\beta=(\psi,\phi,\theta)$ where $\psi:Q\times Q\to \aut{A}$, $\phi:Q\times Q\to \mathrm{End}(A)$ and $\theta:Q\times Q\to A$ satisfy the following conditions
\begin{align*}
\psi_{a,b*c}(\theta_{b,c})+\theta_{a,b*c} & = \psi_{a*b,a*c}(\theta_{a,c})+\phi_{a*b,a*c}(\theta_{a,b})+\theta_{a*b,a*c}\\
 \phi_{a,b*c} &= \phi_{a*b,a*c}\phi_{a,b}+\psi_{a*b,a*c}\phi_{a,c} \\
 \psi_{a,b*c}\psi_{b,c} & = \psi_{a*b,a*c}\psi_{a,c},\label{cocycle}\tag{CC}\\ 
 \psi_{a,b*c}\phi_{b,c} & =  \phi_{a*b,a*c}\psi_{a,b},\\ 
  \phi_{a,a}+\psi_{a,a}&=1\\
  \theta_{a,a}&=0.
\end{align*}
%
%\begin{eqnarray*}\label{eq1}
%\psi(a,b*c)\psi(b,c) & = &\psi(a*b,a*c)\psi(a,c) \\
%\psi(a,b*c)\phi(b,c) & = & \phi(a*b,a*c)\psi(a,b)\\
%\phi(a,b*c) &= &\phi(a*b,a*c)\phi(a,b)+\psi(a*b,a*c)\phi(a,c)\\
%\psi(a,b*c)(\theta(b,c))+\theta(a,b*c) & = & \psi(a*b,a*c)(\theta(a,c))+\phi(a*b,a*c)(\theta(a,b))+\theta(a*b,a*c)\\
%\phi(a,a)+\psi(a,a) & = & 1 \\
%\theta(a,a) & =&0.
%\end{eqnarray*}
We call such conditions {\it cocycle conditions} and the quandle $Q\times_{\beta} A$ where $\beta(a,b,s)(t)=\phi_{a,b}(s)+\psi_{a,b}(t)+\theta_{a,b}$ for every $a,b\in Q/\alpha$ and $s,t\in A$ is called {\it abelian extension of $Q$} (in the literature this construction is also known as {\it quandle modules} \cite{AG}).
%In particular if the extension $Q\times_{(\psi,\phi,\theta)} A$ is latin, then $\phi$ takes value in $\aut{A}$.

If $Q$ is latin, $u\in Q$ and $\beta=(\phi,\psi,\theta)$ is an abelian cocycle  we can define $\gamma(a)=\psi_{a/u,u}^{-1}$ for every $a\in Q$ in order to construct an abelian cocycle as follows:
\begin{eqnarray*}
\sigma(u)(a,b,s)(t)&=&\gamma(a*b)\left(\phi_{a,b}\gamma(a)^{-1}(s)+\psi_{a,b}\gamma(b)^{-1}(t)+\theta_{a,b}\right)=\\
&=&\underbrace{\phi_{(a*b)/u),u}^{-1}\phi_{a,b}\psi_{a/u,u}(s)}_{\phi(u)_{a,b}(s)}+\underbrace{\psi_{(a*b)/u,u}^{-1}\psi_{a,b}\psi_{b/u,u}(t)}_{\psi(u)_{a,b}(t)} +\underbrace{\psi_{(a*b)/u,u}^{-1}(\theta_{a,b})}_{\theta(u)_{a,b}}.
\end{eqnarray*}
In particular, $Q\times_{\sigma(u)} A$ is isomorphic to $Q\times_\beta A$ and $\psi(u)_{a,u}=\psi_{u,u}$ for every $a\in Q$. We call $\sigma(u)=(\psi(u),\phi(u),\theta(u))$ a {\it $u$-normalized cocycle}. 

%\comment{put 6.8, 6.9 e 6.10 together. Just mention that we use the same techniques as in \cite{MeAndPetr}}
%\begin{lemma}
%Let $Q$ be a connected quandle, $u\in Q$ and $(\phi,\psi,\theta)$ an Abelian cocycle with coefficient in $A$. If $\aut{A}$ is Abelian then $\psi(a,a)=\psi(u,u)$ for every $a\in Q$.
%\end{lemma}
%\begin{proof}
%By equation \eqref{eq1}, setting $b=c$ we have:
%\begin{displaymath}
%\psi(ab,ab)\psi(a,b)=\psi(a,b)\psi(b,b).
%\end{displaymath}
%Since $\aut{A}$ is Abelian, $\psi_u(ab,ab)=\psi_u(b,b)$ and so $\psi(h(u),h(u))=\psi(u,u)$ for every $h\in \dis(Q)$. The quandle $Q$ is connected, hence $\psi(a,a)=\psi(u,u)$ for every $a\in Q$.
%\end{proof}
%
%\begin{corollary}\label{cor1norm}
%Let $Q$ be a latin quandle and $(\psi,\phi,\theta)$ a $u$-normalized Abelian cocycle with coefficients in $A$. If $\aut{A}$ is Abelian then $\psi(a,a)=\psi(a,u)=\psi(u,u)$ for every $a\in Q$.
%\end{corollary} 
Using the same techniques developed in \cite{MeAndPetr} it can be showed that normalized cocycles with coefficients in $\mathbb{Z}_p$ have some special properties, taking also advantage from the fact that the automorphism group of $\mathbb{Z}_p$ is abelian (all the computations are straightforward manipulation of the cocycle conditions and then omitted). We define $g,f,h: Q\times Q\longrightarrow Q\times Q$ as:
$$g: (a,b)\to (u*a,u*b), \quad f: (a,b)\to (a*( b/u),a*u),\quad h:(a,b)\to ((b/(u\backslash a))* a,b).$$
 These mappings are permutations of $Q\times Q$ \cite[Section 4]{MeAndPetr}. 
\begin{proposition}\label{invariant under orb}
Let $Q$ be a latin quandle and $\beta=(\psi,\phi,\theta)$ a $u$-normalized abelian cocycle with coefficients in $\mathbb{Z}_p$. 
Then:
\begin{itemize}
\item[(i)] $\psi_{a,u}=\psi_{u,a}=\psi_{u,u}=\psi_{a,a}$;
\item[(ii)] $\psi_{k(a,b)}=\psi_{a,b}$,
\end{itemize}
for every $a,b\in Q$ and every $k\in \langle f,g,h\rangle$.
\end{proposition}
Applying Proposition \ref{invariant under orb} to the unique connected quandle of size $4$, $Q=\aff(\mathbb{Z}_2^2,f)$ we have that $0$-normalized cocycles of $Q$ have the following form
\begin{equation}\label{norm}
\begin{small}
\psi=\begin{bmatrix}
\lambda&\lambda&\lambda&\lambda\\
        \lambda&\lambda& \mu & \mu\\
        \lambda& \mu &\lambda& \mu \\
        \lambda&\mu & \mu & \lambda
\end{bmatrix},\quad \phi=\begin{bmatrix}
1-\lambda&\phi_0&\phi_0&\phi_0\\
       \phi_1&1-\lambda& \phi_2& \phi_3\\
       \phi_1&\phi_3&1-\lambda&\phi_2\\
       \phi_1&\phi_2& \phi_3& 1-\lambda
\end{bmatrix}.
\end{small}
\end{equation}
where $\lambda,\mu\in \aut{\mathbb{Z}_p}$ and $\phi_i\in \mathrm{End}(\mathbb{Z}_p)$ for $0\leq i\leq 3$. Starting from \eqref{norm} and using the cocycle conditions \eqref{cocycle} we have the following: 

\begin{proposition}\label{final on ab ext}
Let $\beta=(\phi,\psi,\theta)$ be an abelian $0$-normalized cocycle of $Q=\aff(\mathbb{Z}_2^2,f)$ with coefficients in $\mathbb{Z}_p$. Then $\psi_{a,b}=\lambda$ and $\phi_{a,b}=1-\lambda$ for every $a,b\in Q$ for some $0\neq \lambda\in \mathbb{Z}_p$ . 
\end{proposition}

If an abelian cocycle $\beta$ is as in Proposition \ref{final on ab ext}, then the kernel of the canonical projection $Q\times_\beta A\longrightarrow Q$ is a central congruence \cite[Proposition 7.5]{CP}. Therefore subdirectly irreducible latin quandles of size $4p$ are not abelian extensions of $\aff(\mathbb{Z}_2^2,f)$.

\section{Non connected LSS quandles}\label{non-connected}
In this section we study non connected LSS quandles. Every subset of a projection subquandle is a subquandle,  so LSS projection quandle have size 2 or 3. From now on we consider just non-projection quandles. We will denote by $\pi_Q$ the orbit decomposition with respect to the action of $\dis(Q)$.
%\begin{lemma}
%Let $Q$ be a LSS quandle. $Q$ is projection if and only if $|Q|=2,3$.
%\end{lemma}
%\begin{proof}
%Easy.
%\end{proof}

\begin{lemma}
Let $Q$ be a finite non-connected non-projecction LSS quandle and $\alpha,\beta\in Con(Q)$. Then: 
\begin{itemize}
\item[(i)] $Q/\alpha$ is strictly simple.
\item[(ii)] $Q/\alpha$ is projection if and only if $\alpha=\pi_Q$. Moreover $Q/\pi_Q\cong \mathcal{P}_2$.
\item[(iii)]$\alpha\wedge \beta=0_Q$.
\end{itemize}
\end{lemma}
\begin{proof}
(i) The proof is the same as in Lemma \ref{lemma_1_on_HSS}(ii).

(ii) By virtue of item(i), $Q/\alpha$ is projection and strictly simple, so $Q/\alpha=\mathcal{P}_2$ and $\pi_Q\leq \alpha$. Since $Q/\pi_Q\cong \mathcal{P}_2$ then $\alpha=\pi_Q$.

(iii) Let $\alpha$, $\beta$ and $\gamma=\alpha\wedge \beta\in Con(Q)$. If both the factors $Q/\alpha$, $Q/\beta$ are projection, then by item (ii) $\gamma=\alpha=\beta=\pi_Q$. Let assume that $Q/\gamma$ is connected. The blocks of $\gamma$ have all the same size and $[a]_\gamma$ is a subquandle of $[a]_\alpha$. So either $\alpha=\beta$ or $\gamma=0_Q$ since the blocks are strictly simple quandles.
%So necessarily $a$ and $b$ are in different components. Since $Q/\beta$ is connected, there exists $h\in \lmlt(Q)$ mapping $[b]_\beta=[b]_\pi$ to $[a]_\beta$. So $a\in h([b]_\beta)=h([a]_\pi)= [b]_\pi$, contraddiction.
% If $[a]_\beta$ is connected, then $[a]_\beta\leq [a]_\pi$ and so $[a]_\gamma=[a]_\pi=\{a\}$ and since all the blocks with respect to $\beta$ have the same size $\beta=0_Q$. So let $[b]_\beta\cong \mathcal{P}_2$ and therefore all the blocks with respect to $\beta$ are $\mathcal{P}_2$. Since $Q/\beta$ is connected, there exists $h\in \lmlt(Q)$ mapping $[a]_\beta=\{a,c\}$ to $[b]_\beta=\{b,d\}$. So either $b$ or $d$ are mapped to $c$, but $c\notin [b]_\pi$. Contraddiction.
\end{proof}

First we characterize the subdirectly reducible non connected LSS quandles.

\begin{proposition}
Let $Q$ be a subdirectly reducible non-connected non-projection LSS quandle. Then $Q\cong \mathcal{P}_2\times R$ where $R$ is a connected strictly simple quandle.
% is one of the following:
%\begin{itemize}
%\item[(i)] $Q\cong 2\times Q/\alpha$.
%\item[(ii)] $Q=Q/\alpha \cup \{x\}$, $Q/\alpha$ is connected and $x$ acts trivially. \comment{check: what is $\alpha$? I think just kill $x$ identifying it with some element of $Q/\alpha$}
%\end{itemize}
\end{proposition}
\begin{proof}
Let $\alpha\in Con(Q)$ such that $\alpha\wedge \pi_Q=0_Q$.
Then $Q$ embeds subdirectly into $\mathcal{P}_2\times Q/\alpha$. If $Q/\alpha\cong \mathcal{P}_2$, then $Q$ is projection. So $Q/\alpha$ is connected and strictly simple and $Q$ embeds into $\mathcal{P}_2\times Q/\alpha$. The blocks with respect to $\alpha$ have all the same size and therefore they have size $2$ since $\alpha$ is a proper congruence of $Q$. So $|Q|=2|Q/\alpha|$ and then $Q\cong \mathcal{P}_2\times Q/\alpha$.
\end{proof}

Let assume that $Q$ is subdirectly irreducible, so $\pi_Q$ is the unique congruence of $Q$. Let us denote the components of $Q$ by $Q_1, Q_2$ and their mutual actions by
\begin{displaymath}
\rho_1: Q_1\to Conj(\aut{Q_2}), \quad \rho_2:Q_2\to Conj(\aut{Q_1}),
\end{displaymath}
where $\rho_i$ are morphisms of quandles. In particular the action need to satisfy
\begin{equation}\label{conditions on actions}
\rho_1({\rho_2(a)(x)})=L_a\rho_1(x) L_a^{-1}, \quad \rho_2(\rho_1(x)(a))=L_x\rho_2(a) L_x^{-1},
\end{equation}
%
%\begin{equation}\label{conditions on actions}
%L_{\rho_1(x)(a)}=\rho_1(x)L_a\rho_1(x)^{-1}, \quad L_{\rho_2(b)(y)}=\rho_2(b)L_y\rho_2(b)^{-1},
%\end{equation}
for every $x\in Q_1$ and $a\in Q_2$, see \cite[Lemma 1.18]{AG}. 

The class of \textit{2-reductive medial} quandles is the class of quandles such that $\pi_Q\leq \lambda_Q$. Such quandles are described by the following construction which can be found in \cite{Medial}: let $\setof{A_i}{i\in I}$ be a family of abelian groups and let $C=\setof{c_{i,j}\in A_j}{i,j\in I}$ such that $c_{i,i}=0$ and $A_i=\langle c_{j,i}, \, j\in I\rangle$ for every $i\in I$ (if $I$ is finite we can arrange the elements of $C$ is a matrix). Then $Q=(\bigcup_{i\in I}A_i,*)$ where $a*b=b+c_{j,i}$ whenever $a\in A_j$ and $b\in A_i$ is a 2-reductive medial quandle and all $2$-reductive medial quandles arise in this way.  

\begin{proposition}\label{medials}
Let $Q$ be a non faithful subdirectly irreducible LSS quandle. Then $\pi_Q=\lambda_Q$ and $Q$ is one of the following medial $2$-reductive quandles: 
$$ \left\lbrace\mathbb{Z}_1\cup \mathbb{Z}_2, \begin{bmatrix}
0 & 1\\
1 & 0
\end{bmatrix}\right\rbrace, \qquad \left\lbrace\mathbb{Z}_2\cup \mathbb{Z}_2, \begin{bmatrix}
0 & 1\\
1 & 0
\end{bmatrix}\right\rbrace.$$
%
%\begin{itemize}
%\item[(i)] $Q\cong \{\mathbb{Z}_1\cup \mathbb{Z}_2, \begin{bmatrix}
%0 & 1\\
%1 & 0
%\end{bmatrix}\}$.
%\item[(ii)] $Q\cong \{\mathbb{Z}_2\cup \mathbb{Z}_2, \begin{bmatrix}
%0 & 1\\
%1 & 0
%\end{bmatrix}\}$.
%\end{itemize}
\end{proposition}
\begin{proof}
Let $Q$ be a non-faithful non-connected LSS quandle. Since $\pi_Q$ is the unique congruence of $Q$ then $\lambda_Q=\pi_Q$ and so $Q$ is a $2$-reductive medial quandle. The components of $Q$ are projection subquandles, so they have at most $2$ elements. If $|Q_1|=1$ and $|Q_2|=2$ then $Q$ is as in (i). If $|Q_1|=|Q_2|=2$ then $Q$ is as in (ii).
\end{proof}

\begin{lemma}\label{trivialaction}
	Let $Q=Q_1\cup Q_2$ be a finite non-connected LSS quandle. If there exists $a\in Q_2$ such that $\rho_1(b)(a)=a$, for every $b\in Q_1$, then $Q_2=\{a\}$.
\end{lemma}
\begin{proof}
	If the action $\rho_1$ fixes $a$, then $Q_1\cup \{a\}$ is a subquandle of $Q$ which contains $Q_1$ as a proper subquandle. Then $Q=Q_1\cup \{a\}$. 
\end{proof}
\begin{proposition}\label{faithful non connected}
Let $Q=Q_1\cup Q_2$ be a faithful subdirectly irreducible LSS quandle. Then $Q$ is one of the following:
\begin{itemize}
\item[(i)] $Q_1$ is a connected strictly simple quandle and $Q_2=\{x\}$ and $x$ acts trivially on $Q_1$. 
\item[(ii)] $Q_1=\aff(\mathbb{Z}_p,-1)$, $Q_2 =\{1,-1\}\cong \mathcal{P}_2$ and  $$(\pm 1 )\ast  a=a\pm 1, \quad a\ast (\pm 1)=\mp 1$$
 for every $a\in Q_1$. 
\item[(iii)] $Q_1$ and $Q_2$ are connected strictly simple quandles of the same size. 
\end{itemize}
\end{proposition}

\begin{proof}
 If the components of $Q$ are both projection quandles, then they have size 1 or 2 and so $\lmlt(Q)$ is abelian. The center of $\lmlt(Q)$ of faithful quandles is trivial, therefore $Q$ is not faithful. So we have the following cases:

(i) Let $Q_1=R$ and $Q_2=\{x\}$ where $R$ is a connected strictly simple quandle. So $L_x|_{R}=f\in \aut{R}$ and
\begin{displaymath}
a\ast (x\ast b)=L_af(b)=(a*x)\ast (a*b)=x\ast (a*b)=f L_a(b),
\end{displaymath}
i.e. $f$ centralizes $L_a$ for every $a\in R$. Then $f=1$. 

(ii) Let $Q=\{x_1,x_2\}\cup R$ where $R=\aff(\mathbb{Z}_p^n,f)$ is a connected strictly simple quandle. So the image of $\rho_2$ has size $1$ ($R$ is simple and $\aut{Q_2}\cong \mathbb{Z}_2$). By virtue of Lemma \ref{trivialaction}, $\rho_2(a)(x_1)=x_2$ and $\rho_2(a)(x_2)=x_1$ for every $a\in Q_2$, and the action of $Q_1$ on $Q_2$ is not trivial. So we have:
\begin{eqnarray}
L_a L_b^{-1} (x_i\ast z)&=&L_a L_b^{-1}(x_i) *L_a L_b^{-1} (z)=x_i\ast L_a L_b^{-1}(z)\label{cond1}.\\
%x_i\ast (x_j \ast a) &=& (x_i \ast x_j)\ast (x_i \ast a)=x_j\ast ( x_i\ast a)\\
a\ast ( x_i  \ast b) &=& (a\ast x_i )\ast (a\ast b)=x_{j} \ast (a\ast b).\label{con2}
\end{eqnarray}
According to \eqref{cond1}, $\rho_1(x_1)$ and $\rho_1(x_2)$ commute with all the generators of $\dis(R)$. Since $\aut{R}\cong \dis(R)\rtimes \mathbb{Z}_{p^n-1}$,  $\dis(R)$ is a maximal abelian subgroup of $\aut{Q}$ then $\rho_1(x_1),\rho_1(x_2)\in \dis(R)$ i.e. $\rho_1(x_1)(a)=a+\lambda$ and $\rho_1(x_2)(a)=a+\sigma$ for every $a\in R$. Equation \eqref{con2} states that $\rho_1(x_2)=L_a \rho_1(x_1) L_a^{-1}$ for all $a\in R$, therefore $\sigma=f(\lambda)$. Moreover
\begin{eqnarray*}
%x_2\ast L_0 (0)&=&\sigma=L_0(x_1\ast 0)=L_0(\lambda)=f(\lambda)\\
f^2(\lambda)=L_0^2 (x_1\ast 0)=\rho_2(0)^2(x_1)\ast 0=x_1\ast 0=\lambda.
\end{eqnarray*}
Then $f=-1$, $n=1$ and $f(\lambda)=-\lambda=\sigma\neq 0$.% otherwise $Q_1\cup \{a\}$ is a subquandle of $Q$. 
Let $\{\pm 1\} \cup_{\lambda} R$ the quandle described above with $\pm 1 * a=a\pm \lambda$ and $a*(\pm 1)=\mp 1$ for every $a\in R$. The map:
\begin{displaymath}
\phi: \{\pm 1\} \cup_{\lambda} R\longrightarrow \{\pm 1\} \cup_{1} R, \quad \pm 1 \mapsto \pm 1,\,  a\mapsto a\lambda^{-1} 
\end{displaymath}
for every $a\in R$ is an isomorphism.

%Both components are strictly simple. If there exists $\alpha\in Con(Q)$ then $Q$ embeds subdirectly into $2\times Q/\alpha$ and $Q/\alpha$ is not projection and $[a]_\alpha=\{a\}$ of $[a]_\alpha=\{a,b\}$. So either $Q\cong Q/\alpha$, $2\times Q/\alpha$ or $Q\cong Q/\alpha \cup \{x\}$ with trivial action.
(iii) Assume that both $Q_1=R$ and $Q_2=S$ are connected strictly simple quandles. So we have morphisms of quandles $\rho:R\to \aut{S}$ and $\sigma: S\to \aut{R}$. Since both $S$ and $R$ are strictly simple the image of $\rho$ and $\sigma$ is either a trivial quandle or they are injective. Assume that one of the two action is constant and given by $f$, then
\begin{displaymath}
a\ast (x\ast b)=L_a f(b)=(a\ast x)\ast (a\ast b)= f L_a(b),
\end{displaymath}
for every $a,b\in R$ and $x\in S$. So $f=1$ and then by Lemma \ref{trivialaction} $|R|=1$. So we can assume that both the mappings $\sigma$ and $\rho$ are injective.\\
According to \cite[Proposition 3.5]{Principal} the connected subquandles of $\aut{R}$ have size equal to the size of $R$. Therefore $|S|=|R|$. %and every strictly simple quandle of size $|R|$ is contained in $\aut{R}$
% and then $R\cong \aff(\mathbb{Z}_p^n, f)$ and $S\cong \aff(\mathbb{Z}_p^n, g)$ where $f$ and $g$ acts irreducibly. In particular the actions defined by $\sigma_y(a)=(1-f)(y)+f(a)$ and $\tau_b(y)=(1-g)(b)+g(y)$ produce a quandle with components $R$ and $S$.
 \end{proof}

\begin{example}
Let $Q_1=\aff(\mathbb{Z}_p^n, f)$ and $Q_2= \aff(\mathbb{Z}_p^n, g)$ be strictly simple quandles. The actions defined by $\rho_1(y)(a)=(1-f)(y)+f(a)$ and $\rho_2(b)(y)=(1-g)(b)+g(y)$ for every $x,y\in Q_1$ and every $a,b\in Q_2$ satisfy \eqref{conditions on actions}. Then for every pair of strictly simple quandles we can construct a quandle $Q=Q_1\cup Q_2$ as in Proposition \ref{faithful non connected}(iii).
\end{example}

\section{Appendix}\label{Sec:appendix}

\subsection*{Matrices and representations}
%In this section we collect some useful lemmas we are using in the rest of the paper.  

%\begin{proposition}\label{sims} \cite{Sims}
%	Let $G$ be a group and let $X\subseteq G$ such that $G/\gamma_1(G)=\langle\setof{ x\gamma_1(G)}{x\in X}\rangle$. Then $\gamma_1(G)/\gamma_2(G)=\langle \setof{[x,y]\gamma_2}{x,y\in X}$. In particular if $G/\gamma_1(G)$ is cyclic, then $\gamma_1(G)=\gamma_2(G)$.
%\end{proposition}

%Given an action $\rho$ we define:
%$Fix(\rho)=\setof{a\in \mathbb{Z}_p^m}{\rho(b)(a)=(a) \text{ for all }b\in \mathbb{Z}_q^n}=\{0\}$

%
%\begin{lemma}\label{char}\comment{how much is this used?}
%Let $G$ be a group, $f\in \aut{G}$ and $H\leq G$ be a normal $f$ invariant subgroup. If $K\leq G/H$ is $f_{/H}$ invariant, then $\pi_H^{-1}(K)$ is a $f$ invariant subgroup.
%\end{lemma}
%\begin{proof}
%Let $\pi_H^{-1}(K)=L$ and $a\in L$. Then $f(a)H\in K$ and so $f(x)\in L$.
%\end{proof}

We collect some basic results about matrices of rank $2$ and faithful representations.

\begin{lemma}\label{eigenvalues of A}
	Let $A\in GL_2(p)$ with prime order $q>2$ where $q$ divides $p^2-1$. 
	\begin{itemize}
		\item[(i)] $A$ is diagonalizable if and only if $q$ divides $p-1$.
		\item[(ii)] $A$ has no eigenvalues if and only if $q$ divides $p+1$.
	\end{itemize}
	If $q=2$ then $A$ is diagonalizable.
\end{lemma}

\begin{proof}
	We have that, either $q$ divides $p-1$ or $q$ divides $p+1$. Considering the action of $A$ on subspaces of dimension one of $\mathbb{Z}_p^2$ we have $p+1=e+nq$, where $e$ is the number of eigenspaces of $A$. So $e$ can not be $1$, and since $q> 2$, $e=0$ if and only if $q$ divides $p+1$ and $e=2$ if and only if $q$ divides $p-1$.\\
	If the order of $A$ is $2$ then either $A$ is diagonal or it has trace zero and determinand $-1$. So the eigenvectors of $A$ are $1$ and $-1$ and $A$ is diagonalizable.
\end{proof}

Recall that a {\it Singer cycle} is a maximal cyclic irreducible subgroup of $GL_n(p)$ and it has size $p^n-1$.

\begin{lemma}\label{normalizer of matrices} \cite[Theorem 2.3.5]{256}
	Let $A\in GL_2(p)$.
	\begin{itemize}
		\item[(i)] If $A$ is diagonalizable but not scalar then $C_{GL_2(p)}(A)$ is the subgroups of diagonal matrices.
		
		\item[(ii)] If $A$ is irreducible then $C_{GL_2(p)}(A)$ is a Singer cycle and the normalizer is isomorphic to $C_{GL_2(p)}(A)\rtimes_{\rho} \mathbb{Z}_2$ and the action $\rho$ is by $p$-powering.
	\end{itemize}
\end{lemma}

\begin{lemma}\label{faithful action}\cite[Theorem 2.3.5]{256}
	Let $G$ be a group and $\rho:G\longrightarrow GL_2(p)$ be a faithful representation.
	\begin{itemize}
		\item[(i)]  If the action $\rho_g$ is irreducible for some $g\in G$ then $C_G(g)$ is cyclic. In particular if $G$ is abelian then it is cyclic. 
		
		\item[(ii)] If $G$ is not abelian then $\rho$ is irreducible.
		
	\end{itemize}
\end{lemma}

%\begin{proposition}\comment{ 7.2 of f09m214notes}
%	Let 
%	$$H\longrightarrow G\longrightarrow K$$
%	be a short exact sequence of finite groups. If $|H|$ and $|K|$ are coprime and $H$ or $K$ are solvable then $G$ is a splitting extension of $H$ by $K$. 
%\end{proposition}

%\comment{Explain the interplay between $\rho$, $Z(\dis(Q)))$ and $\dis(Q)_a$, maybe later}

%\begin{corollary}\label{embedding as quandle}
%Let $Q$ be a connected quandle and $\mu$ be an Abelian congruence with connected blocks. Then $R(\dis_\mu)$ is a subquandle of $[a]^n$ where $n$ is the number of generators of $Q/\mu$.
%\end{corollary}

%
%\begin{lemma}\label{atom then Ab}
%Let $Q$ be a faithful quandle and $\alpha\in Cong(Q)$ be a minimal element. If $Q$ is solvable, then $\alpha$ is Abelian.
%\end{lemma}
%\begin{proof}
%Since $\dis_\alpha > [\dis_\alpha,\dis_\alpha]\in Norm(Q)$, then $\beta=\c{[\dis_\alpha,\dis_\alpha]}<\alpha$, and therefore $\beta=0_Q$. Hence $[\dis_\alpha,\dis_\alpha] = 1$ and so $\alpha$ is Abelian.
%\end{proof} 
\subsection*{The group $G_{p,q}$} \label{Sec:appendix2}

Let us investigate the properties of the group which is relevant for the contents of Section \ref{Sec:pq}.
%turn out to be the displacement group of quandles in $\LSS(p,q,\gamma_Q)$.

%
%In the following we denote by $G_A$ the semidirect product $\mathbb{Z}_p^2\rtimes_{\rho} \mathbb{Z}_q$, where $\rho(1)=A$.

\begin{proposition}\label{uniqueness of the group}
	Let $p,q$ be primes such that $p^2=1\pmod q$. Then there exists a unique non-trivial semidirect product $G_{p,q}=\mathbb{Z}_p^2\rtimes_{\rho} \mathbb{Z}_q$ with $det(\rho(1))=1$ and $|\rho(1)|=q$ up to isomorphism. In particular $Z(G_{p,q})=1$ and $\gamma_1(G_{p,q})=\mathbb{Z}_p^2\times \{0\}$.
\end{proposition}

\begin{proof}
	
	Let $G_A=\mathbb{Z}_p^2\rtimes_{\rho} \mathbb{Z}_q$ such that $\rho(1)=A$, $det(A)=1$ and the order of $A$ is $q$ (note that $A$ is not a scalar matrix, since it has determinant $1$ and order $q>2$).  Let $A,A^\prime$ be such matrices. Using Lemma \ref{eigenvalues of A} we have two cases.\\
	If $q$ divides $p-1$, then $A$ is diagonalizable with eigenvalues $\lambda_1,\lambda_1^{-1}$ and $A^\prime$ is diagonalizable with eigenvalues $\lambda_1^k,\lambda_1^{-k}$ for some $k$.\\ %So $A^\prime =A^k$. \\
	If $q$ divides $p+1$, $A$ and $A^\prime$ acts irreducibly, and by \cite[Theorem 2.3.3]{256}, they generate conjugate subgroups, i.e. $A^\prime$ and $A^k$ are conjugate for some $1\leq k\leq q-1$. \\
	In both cases there exists $\Phi\in GL_2(p)$ such that $\Phi A^\prime=A^k\Phi$. The map
	\begin{displaymath}
	\phi:G_{A^\prime} \longrightarrow G_A, \quad 
		\begin{cases}
	a^\prime\mapsto \Phi(a)\\
	b^\prime\mapsto \Phi(b)\\
	c^\prime\mapsto c^k
	\end{cases} 
	\end{displaymath}
	is a group isomorphism.
	
	The center of this group is $Z(G_A)=Fix(A)\times ker(\rho)=Fix(A)\times \{0\}$, since the action is faithful. The subgroup $Fix(A)$ correspond to the eigenspace of $A$ relative to $1$, which is trivial (if $A$ is diagonalizable and $\lambda_1=1$, then $det(A)=\lambda_2=1$ and then $A=I$). Moreover $\gamma_1(G_A)$ contains the image of $1-A$ which has dimension $2$, so the derived subgroup coincide with the normal summand $\mathbb{Z}_p^2\times \{0\}$.
\end{proof}
We call $G_{p,q}$ the unique group described in Proposition \ref{uniqueness of the group} which admits a presentation as follows:
\begin{equation}\label{presentation of G_A}
G_{p,q}\cong \langle a,b,c \, | \, a^p=b^p=c^q=1, \, cac^{-1}=a^\alpha b^\beta,\, cbc^{-1}=a^\gamma b^\alpha\rangle. 
\end{equation}
where $\alpha,\beta,\gamma,\delta\in \mathbb{Z}_p$ and $$A=\begin{bmatrix}
\alpha & \gamma\\
\beta &\delta
\end{bmatrix} \in GL_2(p).$$ 
%In the following we are using the formula \comment{cite matrices}. 
%\begin{equation}\label{matrix formula}
%A^n=\begin{bmatrix}
%\alpha & \gamma\\
%\beta & \alpha
%\end{bmatrix}^n=\begin{bmatrix}
%y_n-\alpha y_{n-1} & \gamma y_{n-1}\\
%\beta y_n & y_n-\alpha y_{n-1},
%\end{bmatrix}
%\end{equation}
%where $y_n=\sum_{i=0}^{\lfloor \frac{n}{2}\rfloor} \binom{n-i}{i}Tr(A)^{n-2i}\left(-det(A)\right)^i$ \comment{cite}.
\begin{proposition}\label{aut of G_A}
	The automorphisms of $G_{p,q}$ are given by
	\begin{equation}\label{formula for aut}
	h = 	\begin{cases}
	h|_{\gamma_1(G_{p,q})}=H\in GL_2(p),\quad c\mapsto  c^{\pm 1} d_h       
	\end{cases} 
	\end{equation}
	for some $d_h\in \langle a,b\rangle$ and $HAH^{-1}=A^{\pm 1}$. In particular, 
	\begin{displaymath}
	|\aut{G_{p,q}}|=
		\begin{cases}
	2p^2(p-1)^2, \, \text{if $q$ divides $p-1$}\\
	2p^2(p-1)(p+1), \, \text{if $q$ divides $p+1$}.
	\end{cases}.
	\end{displaymath}
	%Therefore \begin{equation}\label{aut size}
	%|\aut{G_{p,q}}| = \left\{
	%    \begin{array}{l} 2p^2(p-1)^2,\, \,\, \quad \qquad \text{ if $q$ divides $p-1$}\\
	%2p^2(p-1)(p+1),\quad \text{ if $q$ divides $p+1$}
	%\end{array}\right.
	%\end{equation}
\end{proposition}
\begin{proof}
	An automorphism of $G_{p,q}$ is uniquely determined by its image on $a,b,c$ and according to the presentation in \eqref{presentation of G_A} any automorphism $h\in \aut{G_{p,q}}$ is given by $h|_{\gamma_1(G_{p,q})}=H\in GL_2(p)$, and $h(c)=c^{u_h}d_h$ with $u_h\in \{1,\ldots,q-1\}$ and $d_h\in \mathbb{Z}_p^2$. On the other hand such a triple $(H,u_h,d_h)$ defines an automorphism of $G_{p,q}$ if and only if the following two conditions hold:
	\begin{eqnarray}
	HA=A^{u_h} H\label{matrix_eq}\\
	h(c^q)=1.
	\end{eqnarray}
	We consider two cases according to Lemma \ref{eigenvalues of A}. If $q$ divides $p-1$ then $A$ is diagonalizable, and using a basis of eigenvectors of $A$ an easy matrix computation shows that if $H$ satisfies \eqref{matrix_eq}, then either $u_h= 1$ and $H$ is diagonal or $u_h=-1$ and $H$ has zero entries on the diagonal ($A$ is not scalar). Therefore there exists $2(p-1)^2$ such matrices. \\% by Lemma \ref{normalizer of matrices}(i) 
	If $q$ divides $p+1$, $H$ satisfies \eqref{matrix_eq} if and only if $H$ belongs to the normalizer of $A$. In such case, the action of the elements of the normalizer is by $p$-powering. Therefore, $u_h=\pm 1$, since $p=-1$ mod $q$ and the size of the normalizer is $2(p-1)(p+1)$ (see Lemma \ref{normalizer of matrices}(ii)). So $h$ is defined as in \eqref{formula for aut}. It is easy to check that:
	\begin{displaymath}
	h(c)^q=(c^{\pm 1} d_h)^q = c^q \prod_{j=0}^{q-1} \left[A^{\pm j}(d_h)\right]= \prod_{j=0}^{q-1} \left[A^{ j}(d_h)\right].
	\end{displaymath}
	Switching to additive notation, since $(1-A)\sum_{j=0}^{q-1} A^{ j}=1-A^{ q}=0$ and $1-A\in GL_2(p)$, then $\sum_{j=0}^{q-1} A^{j}=0$, i.e $h(c^q)=1$. So every mapping defined as in \eqref{formula for aut} preserves the relations between the generators of $G_{p,q}$ for all the $p^2$ choices of $d_h$.
	%Therefore for each choice of $H$ there are $p^2$ choices of $d_h$, therefore formulas \eqref{aut size} follow.
	%
	%If $q$ divides $p-1$ 
	%
	%
	%$$A=\begin{bmatrix} \lambda & 0\\
	%0 & \lambda^{-1}
	%\end{bmatrix},\quad B=\begin{bmatrix}
	%0 & -1\\
	%1 & -a
	%\end{bmatrix},\quad C=\begin{bmatrix}
	%x & z\\
	%y & t
	%\end{bmatrix}$$
	%Then 
	%\begin{itemize}
	%\item[(i)] $CA-AC=0$ if and only if $y=z=0$. 
	%\item[(ii)] $CA-A^{-1}C=0$ if and only if $t=x=0$.
	%%\item[(iii)] $CB-BC=0$ if and only if $y+z=0$ and $t=x-ay$. 
	%%\item[(iv)] $CB-B^{-1}C=0$ if and only if $Tr(C)=0$ and $z=y-ax$.
	%\end{itemize}
	%In particular the number of non-singular matrix satisfying (i) or (ii) is $2(p-1)^2$. Non-singular matrix satisfying $CB=B^{\pm 1}C$ are exactly the elements of the normalizer of $B$ which has size $2(p-1)(p+1)$.
\end{proof}
%The following corollaries can be easily checked using the characterization of automorphism given in Proposition \ref{aut of G_A}. 
\begin{remark}\label{existence of aut with desired prop}
	Let $f\in \aut{G_{p,q}}$ with $u_f=-1$.  
	\begin{itemize}
		\item[(i)] Normal subgroups of $G_{p,q}$ contained in the derived subgroup are the eigenspaces of $A$. If the subgroup $Fix(f)$ is normal then it is trivial. Indeed $FAF^{-1}=A^{-1}$ and if $A(x)=k x$ for $x\in Fix(f)$ and some $0\neq k\in \mathbb{Z}_p$ then 
		\begin{displaymath}
		A^{-1}(x)=A^{-1}F(x)=FA(x)=kx=A(x)
		\end{displaymath}
		and so $A^2(x)=x$. Then $A(x)=x$ since the non-trivial orbits with respect to the action of $A$ has size $q>2$. Then $x\in Z(G)=1$.
		
		\item[(ii)] If $Tr(F)=0$ and $det(F)=-1$, the $F$ has eigenvalues $1$ and $-1$ with eigenvectors $a$ and $b$. Let $H=Fix(f)=\langle a\rangle$, the mapping
		\begin{equation}\label{right mult}		
		\varphi: G/H\longrightarrow G/H, \quad gH\mapsto gf(g)^{-1} H
		\end{equation}
		is bijective. Assume that $gH=c^n b^m H$ and $hH=c^l b^s H$ and $\varphi(gH)=\varphi(hH)$. Then $n=l$ since $Fix(f_{\gamma_1(G)})=1$ and so proceeding by equivalences we have
		\begin{eqnarray*}
			b^m f(b^{-m} c^{-n})H &=& b^s f(b^{-s} c^{-n})H\\
			b^{2m} (c^{-1}d)^{-n}H &=& b^{2s} (c^{-1}d)^{-n} H\\
			(c^{-1}d)^{n} b^{2(m-s)} (c^{-1}d)^{-n} &\in & H.
		\end{eqnarray*} 
		Switching to additive notation and setting $z=2(m-s)b$ we have that $A^n(z)=ka$ for some $k\in \mathbb{N}$. Using that $FA=A^{-1}F$ we get
		\begin{displaymath}
		FA^n(z)=F(ka)=ka=A^{-n}F(z)=A^{-n}(-z),
		\end{displaymath}
		and therefore $A^{2n}(ka)=A^n(-z)=-ka$. If $k\neq 0$ then $Fix(f)$ is an eigenspace of $A$, so according to (i) $k=0$ and then $s=m$. 
		%Then $A^{4n}(ka)=ka$ and since $Fix(f)$ is not an eigenspace of $A$ (since it is not normal, see item (i)) then $A^{n}=1$ and so $n=0 \pmod q$. Hence $m=s$ and $\varphi$ is injective.
		%\item[(iii)] The following automorphism of $G_{p,q}$ 
		%\begin{displaymath}
		%f=\left\{
		%    \begin{array}{l}
		%F=\begin{bmatrix}
		%0 & 1\\
		%1 & 0
		%\end{bmatrix},	\quad c\mapsto c^{-1} d_f
		%\end{array},\right. 
		%\end{displaymath}
		%satisfies $Tr(F)=0$ $det(F)=-1$.
	\end{itemize}
\end{remark}

\subsection*{The group $G_k$}
The group $\mathcal{Q}_8$ has the following presentation 
$$\langle x,y,z\,|\, x^2=y^2=[x,y]=z,\, z^2=[z,y]=[z,x]=1\rangle.$$
In this section we investigate the group $G_k=\mathbb{Z}_p^2\rtimes_{\rho} \mathcal{Q}_8$, where $p=1\pmod 3$ and $k\in\mathbb{Z}_p$ such that $k^3=1$. The action $\rho$ is the irreducible action of $\mathcal{Q}_8$ defined by setting
\begin{displaymath}
\rho_x=\begin{bmatrix}
0 & -1\\
1 & 0
\end{bmatrix},\quad
\rho_{y}=\begin{bmatrix}
k^2 & k\\
k & -k^2
\end{bmatrix},\quad \rho_z=\begin{bmatrix}
-1 & 0\\
0 & -1
\end{bmatrix}.
\end{displaymath}
According to \cite{Mayr} $\rho$ is the unique irreducible action of $\mathcal{Q}_8$ of degree $2$, and moreover it is a fixed-point-free action, i.e. $Fix(\rho_h)=\setof{v\in \mathbb{Z}_p^2}{\rho_h(v)=v}=0$ for every $h\in \mathcal{Q}_8$. Equivalently $1-\rho_h\in GL_2(p)$ for every $h\in \mathcal{Q}_8$. In particular $\gamma_2(G_k)=\mathbb{Z}_p^2\times \{1\}$ and $\gamma_1(G)=\mathbb{Z}_p^2\rtimes Z(\mathcal{Q}_8)$. 
\begin{proposition}
The automorphisms of $G_k$ are given by the mappings
\begin{equation}\label{aut of Gk}
f=\begin{cases} x\mapsto h v_1,\quad y\mapsto gv_2,\quad z\mapsto zw,\quad F=f|_{\gamma_2(G_k)}\in GL_2(p) \end{cases}
\end{equation}
where $g,h\in \mathcal{Q}_8$ and $v_1,v_2,w\in \mathbb{Z}_p^2$ such that
\begin{displaymath}
\begin{cases}
f_{\gamma_2(G_k)}=\begin{cases} x\mapsto h\quad y\mapsto g,\quad z\mapsto z, \quad \text{is an automorphism of } \mathcal{Q}_8 \end{cases}\\
F\rho_x=\rho_h F, \\ 
F\rho_y=\rho_g F, \\
v_1=2^{-1}(1+\rho_h)(w),\\ 
v_2=2^{-1}(1+\rho_g)(w)
%w=(1-\rho_h)(v_1)=(1-\rho_g)(v_2)=[(1-\rho_h)(1-\rho_g)^{-1}+(1-\rho_g)(1-\rho_h)^{-1}](w).
\end{cases}
\end{displaymath}
 In particular $|\aut{G_k}|=24(p-1)p^2$.
\end{proposition}

\begin{proof}
If $f\in \aut{G_k}$ then $f_{\gamma_2(G_k)}\in \aut{\mathcal{Q}_8}$ and then $h,g$ do not commute. On the other hand assume that $f_{\gamma_2(G_k)}$ is an automorphism of $\mathcal{Q}_8$, then the mapping $f$ as in \eqref{aut of Gk} is an automorphism of $G_k$ if and only if it respects the relation between the generators of $G_k$. Hence for every $v\in \mathbb{Z}_p^2$ we have
\begin{equation}\label{condition on F}
\begin{cases}
f(xvx^{-1})=F\rho_x(v)=f(x)f(v)f(x)^{-1}=\rho_hF(v)\\
f(yvy^{-1})=F\rho_y(v)=f(y)f(v)f(y)^{-1}=\rho_gF(v)\\
\end{cases}.
\end{equation}
Moreover the mapping $f$ has to satisfy
\begin{equation}\label{condition 1}
\begin{cases} f(x)^2=h v_1 h v_1=h^2 \rho_h^{-1}(v_1)v_1=zw=f(z)\\
 f(y)^2=g v_2 g v_2=g^2 \rho_g^{-1}(v_2)v_2=zw=f(z)\\
 [f(x),f(y)]=v_1^{-1} h^{-1} v_2^{-1} g^{-1} h v_1 g v_2=f(z)=zw.
\end{cases}
\end{equation}
After a straightforward manipulation and using that $\rho_h$ is fixed-point-free and that $(1-\rho_h)^{-1}=2^{-1}(1+\rho_h)$ for every non central element $h\in \mathcal{Q}_8$ we have that \eqref{condition 1} is equivalent to
\begin{equation}
\begin{cases} v_1=2^{-1}(1+\rho_h)(w)\\
v_2=2^{-1}(1+\rho_g)(w)\\
w=2^{-1}\left((1-\rho_h)(1+\rho_g)+(1-\rho_g)(1+\rho_h)\right)(w).
\end{cases}
\end{equation}
Using that $\rho_g\rho_h+\rho_h\rho_g=0$ we have that the last condition is trivially satisfied by any $w\in \mathbb{Z}_p^2$. The elements $v_1$ and $v_2$ are uniquely determined by $w$.\\
The mapping
\begin{equation}\label{homo}
\aut{G_k}\longrightarrow \aut{\mathcal{Q}_8},\quad f\mapsto f_{\gamma_2(G_k)}
\end{equation}
is a surjective group homomorphism. Indeed setting $w=v_1=v_2=0$ \eqref{condition on F} has a solution for every $f_{\gamma_2(G_k)}\in \aut{\mathcal{Q}_8}$. The kernel of the homomorphism \eqref{homo} has size $(p-1)p^2$, since \eqref{condition on F} is satisfied just by diagonal matrices and $v_1, v_2$ are uniquely determined by the choice of $w$. Therefore $|\aut{G_k}|=|\aut{\mathcal{Q}_8}|(p-1)p^2$.
\end{proof}

\noindent In the following Lemma we account the number of the relevant automorphisms of $G_k$ as addressed in Theorem \ref{4p iff}. Note that the if $f\in \aut{G_k}$ has order $3$ and the induced automorphism $f_{\gamma_1(G_k)}\neq 1$ then it has also order $3$ and therefore it acts irreducibly on $G_k/\gamma_1(G_k)$.
\begin{lemma}\label{aut of Gk_2}
The subset of $\aut{G_k}$  
$$\mathcal{F}=\setof{f\in \aut{G_k}}{|f|=3, \, f_{\gamma_1(G_k)}\neq 1, \, |Fix(f)|=2p}$$ 
has size $16p$.
\end{lemma}
\begin{proof}
If $f\in \mathcal{F}$ then the order of $f_{\gamma_1(G_k)}$ is $3$ and so it is one of the two following matrices:
\begin{displaymath}
		A=\begin{bmatrix}
	0 & 1\\
	1 & 1
	\end{bmatrix}, \qquad B=\begin{bmatrix}
	1 & 1\\
	1 & 0
	\end{bmatrix}.
	\end{displaymath}
  Let $f_{\gamma_1(G_k)}=A$, then the image of $f$ of the generators of $G_k$ is  
  	\begin{displaymath}
	f=\begin{cases}
	x \mapsto yz^s v_1, \quad 	y \mapsto xyz^l v_2,\quad
	z \mapsto z w, \quad f|_{\gamma_1(G_k)} = F
	\end{cases}
	\end{displaymath}
	for $s,l=0,1$ and $v_1,v_2,w\in \gamma_2(G_k)$.	The order of $F$ is $3$ and so its eigenvalues are $1$ and $\lambda$ and the order of $\lambda$ is $3$. So $\lambda=k^t$ where $t=1,2$ and $det(F)=Tr(F)-1=\lambda$. The matrix $F$ has to satisfy the following conditions
	\begin{equation}\label{condition on F of order 3}
	F\rho_x=(-1)^s\rho_yF, \quad F\rho_y=(-1)^l \rho_x \rho_y F.
	\end{equation}	
	A case-by-case discussion on the values of $s$ and $l$ shows that for each choice of $s,l$ there exists just one solution to \eqref{condition on F of order 3}. Moreover, since $Fix(f)\leq \gamma_1(G_k)$ ($f_{\gamma_1(G_k)}$ acts irreducibly) there exists $v\in \gamma_2(G_k)$ such that $f(zv)=zv$. Then $f(zv)=zwf(v)=zv$ holds if and only if $w=(1-F)(v)$. Hence there exists $p$ choices for $w$, since the image of $1-F$ has dimension $1$. So there exists $8p$ such automorphisms.\\
	Note that $f\in \mathcal{F}$ if and only if $f^2\in \mathcal{F}$ and moreover $f_{\gamma_1(G_k)}=A$ if and only if $f_{\gamma_1(G_k)}^2=B$. Therefore there are $8p$ automorphisms in $\mathcal{F}$ such that the induced automorphism over $G_k/\gamma_1(G_k)$ is $B$. Hence the size of $\mathcal{F}$ is $16p$.
\end{proof}

\begin{lemma}\label{centralizer}
Let $f(\lambda)\in \aut{G_k}$ given by
\begin{equation*}		f(\lambda)=\begin{cases}
		x \mapsto y, 	\quad y \mapsto xy,\quad
		  z\mapsto z,\quad 		
		F(\lambda)={f(\lambda)}|_{\gamma_2(G_k)} =\begin{bmatrix}
		 	-k(1+\lambda) & -(1+\lambda)\\
		 0&-k^2(1+\lambda)  	\end{bmatrix}	
				\end{cases}
		\end{equation*}
	where $\lambda=k,k^2$. Then $|C_{\aut{G_k}}(f(\lambda))|=3p(p-1)$.
	\end{lemma}

\begin{proof}
If $h\in C_{\aut{G_k}}(f(\lambda))$ then $h_{\gamma_2(G_k)}$ centralizes $f(\lambda)_{\gamma_2(G_k)}$ and therefore $h_{\gamma_2(G_k)}$ is a power of $f(\lambda)_{\gamma_2(G_k)}$. Hence 
\begin{equation*}
h=\begin{cases} x\mapsto t v_1,\quad y\mapsto gv_2,\quad z\mapsto zw,\quad H=h|_{\gamma_2(G_k)}\in GL_2(p) \end{cases}
\end{equation*}
and $(t,g)$ is $(x,y)$ or $(y,xy)$ or $(xy,x)$. \\
If $(t,g)=(x,y)$ then 
\begin{equation}\label{comm condition}
\begin{cases}
hf(\lambda)(x)=h(y)=yv_2=f(\lambda) h(x)=yF^\lambda(v_1)\\
hf(\lambda)(y)=h(xy)=xv_1 y v_2=f(\lambda) h(y)=xy F^\lambda(v_2)\\
hf(\lambda)(z)=h(z)=zw=f(\lambda) h(z)=zF(\lambda)(w)\\
\end{cases}  \Leftrightarrow \quad
%\begin{cases}
%v_2=F(\lambda)(v_1)\\
%v_1 y v_2=y F(\lambda)(v_2)\\
%w=F(\lambda)(w)\\
%\end{cases}\quad 
\begin{cases}
v_2=F(\lambda)(v_1)\\
(1- F(\lambda))(v_2)-\rho_y (v_1) =0\\
w=F(\lambda)(w)\\
\end{cases}
\end{equation}
Moreover $F$ is a diagonal matrix and therefore it centralizes $F(\lambda)$. Using that $v_1=(1-\rho_x)^{-1}(w)=2^{-1}(1+\rho_x)(w)$ we have that every choice of $w\in Fix(F(\lambda))$ provides a solution to \eqref{comm condition} and so it defines an element of the centralizer of $f(\lambda)$. Hence there are $p(p-1)$ such automorphisms.\\
Similarly it can be proved that there exists $p(p-1)$ automorphisms in the centralizer of $f(\lambda)$ for the other choices of $t$ and $g$. Hence the size of the centralizer is $3p(p-1)$.
%
%If $(h,g)=(y,xy)$ then
%\begin{equation}
%\begin{cases}
%ff^\lambda(x)=f(y)=xyv_2=f^\lambda f(x)=xyF^\lambda(v_1)\\
%ff^\lambda(y)=f(xy)=yv_1 xy v_2=f^\lambda f(y)=f^\lambda(xy v_2)=yxyF^\lambda(v_2)\\
%ff^\lambda(z)=f(z)=zw=f^\lambda f(z)=zF^\lambda(w)\\
%\end{cases} \quad \Leftrightarrow \quad \begin{cases}
%v_2=F^\lambda(v_1)\\
%(1- F^\lambda)(v_2)-\rho_{xy} (v_1) =0\\
%w=F^\lambda(w)\\
%\end{cases}
%\end{equation}
%\comment{Need to prove that for every choice of $w\in Fix(F^\lambda)$ we have an element of the centralizer}
%Moreover $F$ centralizes $F^\lambda$ if and only if
%\begin{equation}
%F=\begin{bmatrix}
%-(k+1)t & kt  \\
% 0 & t
%\end{bmatrix}
%\end{equation}
%for $t\neq 0$. Hence there are $p(p-1)$ such automorphisms.\\
%If $(h,g)=(xy,x)$ then
%\begin{equation}
%\begin{cases}
%ff^\lambda(x)=f(y)=xv_2=f^\lambda f(x)=yxyF^\lambda(v_1)\\
%ff^\lambda(y)=f(xy)=xyv_1 x v_2=f^\lambda f(y)=y F^\lambda(v_2)\\
%ff^\lambda(z)=f(z)=zw=f^\lambda f(z)=zF^\lambda(w)\\
%\end{cases}\quad \Leftrightarrow \quad \begin{cases}
%v_2=F^\lambda(v_1)\\
%(1- F^\lambda)(v_2)-\rho_{y} (v_1) =0\\
%w=F^\lambda(w)\\
%\end{cases}
%\end{equation}
%\comment{Need to prove that for every choice of $w\in Fix(F^\lambda)$ we have an element of the centralizer}
%Moreover $F$ centralizes $F^\lambda$ if and only if
%\begin{equation}
%F=\begin{bmatrix}
%t	& -kt\\
%0	& k^2t
%\end{bmatrix}
%\end{equation}
%for $t\neq 0$. Hence there are $p(p-1)$ such automorphisms.
\end{proof}

\bibliographystyle{amsalpha}
\bibliography{references}

\newcommand{\etalchar}[1]{$^{#1}$}
\def\cprime{$'$} \def\cprime{$'$}
\providecommand{\bysame}{\leavevmode\hbox to3em{\hrulefill}\thinspace}
\providecommand{\MR}{\relax\ifhmode\unskip\space\fi MR }
\providecommand{\MRhref}[2]{%
  \href{http://www.ams.org/mathscinet-getitem?mr=#1}{#2}
}
\providecommand{\href}[2]{#2}
\begin{thebibliography}{JPSZD15}

\bibitem[AG03]{AG}
Nicol{\'a}s Andruskiewitsch and Mat{\'{\i}}as Gra{\~n}a, \emph{From racks to
  pointed {H}opf algebras}, Adv. Math. \textbf{178} (2003), no.~2, 177--243.
  \MR{1994219 (2004i:16046)}

\bibitem[BB19]{GB}
Giuliano {Bianco} and Marco {Bonatto}, \emph{{On connected quandles of prime
  power order}}, arXiv e-prints (2019), arXiv:1904.12801.

\bibitem[Ber12]{UA}
Clifford Bergman, \emph{Universal algebra}, Pure and Applied Mathematics (Boca
  Raton), vol. 301, CRC Press, Boca Raton, FL, 2012, Fundamentals and selected
  topics. \MR{2839398}

\bibitem[BF67]{modules2}
L.~Brickman and P.~A. Fillmore, \emph{The invariant subspace lattice of a
  linear transformation}, Canadian Journal of Mathematics \textbf{19} (1967),
  810–822.

\bibitem[Bia15]{GiuThe}
Giuliano Bianco, \emph{On the transvection group of a rack}, Ph.D. thesis,
  Universit\'a degli studi di Ferrara, 2015.

\bibitem[Bon20]{Principal}
Marco Bonatto, \emph{Principal and doubly homogeneous quandles}, Monatshefte
  f{\"u}r Mathematik \textbf{191} (2020), no.~4, 691--717.

\bibitem[BS19]{CP}
Marco {Bonatto} and David {Stanovsk{\'y}}, \emph{{Commutator theory for racks
  and quandles}}, arXiv e-prints (2019), arXiv:1902.08980.

\bibitem[BV18]{MeAndPetr}
Marco Bonatto and Petr Vojt\v{e}chovsk\'{y}, \emph{Simply connected latin
  quandles}, J. Knot Theory Ramifications \textbf{27} (2018), no.~11, 1843006,
  32. \MR{3868935}

\bibitem[CEH{\etalchar{+}}13]{C3}
W.~Edwin Clark, Mohamed Elhamdadi, Xiang-dong Hou, Masahico Saito, and Timothy
  Yeatman, \emph{Connected quandles associated with pointed abelian groups},
  Pacific J. Math. \textbf{264} (2013), no.~1, 31--60. \MR{3079760}

\bibitem[CH13]{C4}
W.~Edwin Clark and Xiang-dong Hou, \emph{Galkin quandles, pointed abelian
  groups, and sequence {A}000712}, Electron. J. Combin. \textbf{20} (2013),
  no.~1, Paper 45, 8. \MR{3040607}

\bibitem[EMR10]{Elham}
M.~{Elhamdadi}, J.~{MacQuarrie}, and R.~{Restrepo}, \emph{{Automorphism groups
  of Quandles}}, arXiv e-prints (2010), arXiv:1012.5291.

\bibitem[ESG01]{EGS}
Pavel Etingof, Alexander Soloviev, and Robert Guralnick, \emph{Indecomposable
  set-theoretical solutions to the quantum {Y}ang-{B}axter equation on a set
  with a prime number of elements}, J. Algebra \textbf{242} (2001), no.~2,
  709--719. \MR{1848966 (2002e:20049)}

\bibitem[ESS99]{ESS}
Pavel Etingof, Travis Schedler, and Alexandre Soloviev, \emph{Set-theoretical
  solutions to the quantum {Y}ang-{B}axter equation}, Duke Math. J.
  \textbf{100} (1999), no.~2, 169--209. \MR{1722951 (2001c:16076)}

\bibitem[FM87]{comm}
Ralph Freese and Ralph McKenzie, \emph{Commutator theory for congruence modular
  varieties}, London Mathematical Society Lecture Note Series, vol. 125,
  Cambridge University Press, Cambridge, 1987. \MR{909290}

\bibitem[Fri11]{modules}
Harald Fripertinger, \emph{The number of invariant subspaces under a linear
  operator on finite vector spaces}, Adv. Math. Commun. \textbf{5} (2011),
  no.~2, 407--416. \MR{2801603}

\bibitem[Gra04]{Grana_p2}
Mat{\'{\i}}as Gra{\~n}a, \emph{Indecomposable racks of order {$p^2$}},
  Beitr\"age Algebra Geom. \textbf{45} (2004), no.~2, 665--676. \MR{2093034
  (2005k:57025)}

\bibitem[Hou12]{Hou}
Xiang-Dong Hou, \emph{Finite modules over {$\Bbb Z[t,t^{-1}]$}}, J. Knot Theory
  Ramifications \textbf{21} (2012), no.~8, 1250079, 28. \MR{2925432}

\bibitem[HSV16]{HSV}
Alexander Hulpke, David Stanovsk\'{y}, and Petr Vojt\v{e}chovsk\'{y},
  \emph{Connected quandles and transitive groups}, J. Pure Appl. Algebra
  \textbf{220} (2016), no.~2, 735--758. \MR{3399387}

\bibitem[Joy82]{J}
David Joyce, \emph{A classifying invariant of knots, the knot quandle}, J. Pure
  Appl. Algebra \textbf{23} (1982), no.~1, 37--65. \MR{638121 (83m:57007)}

\bibitem[JPSZD15]{Medial}
P\v{r}emysl Jedli\v{c}ka, Agata Pilitowska, David Stanovsk\'{y}, and Anna
  Zamojska-Dzienio, \emph{The structure of medial quandles}, J. Algebra
  \textbf{443} (2015), 300--334. \MR{3400403}

\bibitem[Mat82]{Matveev}
S.~V. Matveev, \emph{Distributive groupoids in knot theory}, Mat. Sb. (N.S.)
  \textbf{119(161)} (1982), no.~1, 78--88, 160. \MR{672410}

\bibitem[{McC}12]{McC}
James {McCarron}, \emph{{Connected Quandles with Order Equal to Twice an Odd
  Prime}}, Oct 2012, p.~arXiv:1210.2150.

\bibitem[Pet00]{Mayr}
Mayr Peter, \emph{Fixed-point-free representations over fields of prime
  characteristic}, Technical Report of the Mathematics Department -
  Univesit\"at Linz \textbf{554} (2000).

\bibitem[RIG]{RIG}
Mat{\'{\i}}as Gra{\~n}a and Leandro Vendramin, \emph{{Rig, a GAP package for
  racks, quandles and Nichols algebras}}.

\bibitem[Sho92]{256}
M.~W. Short, \emph{The primitive soluble permutation groups of degree less than
  {$256$}}, Lecture Notes in Mathematics, vol. 1519, Springer-Verlag, Berlin,
  1992. \MR{1176516}

\bibitem[Sim94]{Sims}
Charles~C. Sims, \emph{Computation with finitely presented groups},
  Encyclopedia of Mathematics and its Applications, vol.~48, Cambridge
  University Press, Cambridge, 1994. \MR{1267733 (95f:20053)}

\bibitem[Sta15]{Stanos}
David Stanovsk\'{y}, \emph{A guide to self-distributive quasigroups, or {L}atin
  quandles}, Quasigroups Related Systems \textbf{23} (2015), no.~1, 91--128.
  \MR{3353113}

\bibitem[Voj18]{BruckPetr}
Petr Vojt\v{e}chovsk\'{y}, \emph{Bol loops and {B}ruck loops of order {$pq$} up
  to isotopism}, Finite Fields Appl. \textbf{52} (2018), 1--9. \MR{3807838}

\bibitem[Win72]{winter1972}
David~L. Winter, \emph{The automorphism group of an extraspecial $p$-group},
  Rocky Mountain J. Math. \textbf{2} (1972), no.~2, 159--168.

\end{thebibliography}

\end{document}